\documentclass[11pt,reqno]{amsart}
\usepackage{amssymb, a4wide, amsmath, mathtools,xy}
\xyoption{all}
\usepackage{latexsym,pdfsync,xcolor,graphicx}
\usepackage{multirow}

\usepackage[T1]{fontenc}

\usepackage[utf8]{inputenc}

\usepackage{array}
\usepackage{upgreek}
\usepackage{pstricks-add}
\usepackage{enumerate}
\usepackage{orcidlink}
\usepackage[english]{babel}	
\usepackage{comment}
\allowdisplaybreaks

\usepackage{pgf,tikz}
\usepackage{wrapfig}
\usetikzlibrary{arrows}

\usepackage{float}
\usepackage{appendix}

\usepackage{hyperref}
\hypersetup{colorlinks,
    citecolor=black,
    filecolor=black,
    linkcolor=black,
    urlcolor=black}

\theoremstyle{plain}%
\newtheorem{theorem}{Theorem}[section]
\newtheorem{lemma}[theorem]{Lemma}%
\newtheorem{corollary}[theorem]{Corollary}%

\theoremstyle{definition}
\newtheorem{definition}[theorem]{Definition}%
\newtheorem{example}[theorem]{Example}%

\theoremstyle{remark}
\newtheorem{remark}[theorem]{Remark}

\newcommand{\ot}{\otimes}

\title[On modules over a Hopf brace]{On modules over a Hopf brace}

\begin{document}

\maketitle

\begin{center}
	{\bf R.
		Gonz\'{a}lez Rodr\'{\i}guez$^{1,2}$, B.  Ramos P\'erez$^{1,3}$ and A.B. Rodríguez Raposo$^{1,4}$ }.
\end{center}

\vspace{0.4cm}

\begin{center}	{\small $^{1}$ CITMAga, 15782 Santiago de Compostela, Spain.}
\end{center}
\vspace{0.2cm}
\begin{center}
	{\small $^{2}$ Universidade de Vigo, Departamento de Matem\'{a}tica Aplicada II, E.E. Telecomunicaci\'on,
		E-36310 Vigo, Spain.
		\\email: rgon@dma.uvigo.es}
\end{center}
\vspace{0.2cm}
\begin{center}
	{\small $^{3}$  Universidade de Santiago de Compostela. Departamento de Matem\'aticas,  Facultade de Matem\'aticas, E-15782 Santiago de Compostela, Spain. 
		\\email: braisramos.perez@usc.es}
\end{center}
\vspace{0.2cm}
\begin{center}
	{\small $^{4}$  Universidade de Santiago de Compostela, Departamento de Didácticas Aplicadas,  Facultade de C.C. Educación, E-15782 Santiago de Compostela, Spain.
		\\email: anabelen.rodriguez.raposo@usc.es}
\end{center}
\vspace{0.2cm}

\begin{abstract}
     Let $\mathbb{H}=(H_{1},H_{2})$ be a Hopf brace in a symmetric monoidal category ${\sf C}$. In this article it is proved that the category of modules over $\mathbb{H}$ is isomorphic to the category of modules over the smash product algebra $H_{1}\sharp H_{2}$. Furthermore, the category of modules over $\mathbb{H}$ in the sense of Zhu is characterized by the condition that a certain action lies in the cocommutativity class of $H_{2}$.    
 \end{abstract}
 
 

\section{Introduction}
The Quantum Yang-Baxter Equation (QYBE) was introduced by Chen~N.~Yang \cite{Yang} and Rodney~J.~Baxter \cite{Bax} between the late 1960s and the early 1970s in order to study integrable systems in the field of quantum physics and statistical mechanics. At present, it is considered as one of the most important equations in the field of mathematical physics, serving as a bridge between abstract algebra and theoretical physics. Despite the simplicity of its formulation, obtaining a complete classification of the solutions to the QYBE is by no means a simple problem and it is still open nowadays. In an effort to tackle it, Drinfeld \cite{Drin} proposed to study the simplest class of solutions, the set-theoretical ones.

In this line, L. Guarnieri and L. Vendramin \cite{GV} introduce the notion of skew brace as a non-abelian version of the brace structure previously given by W. Rump \cite{Rump}, proving that skew braces induce non-degenerate and non-necessarily involutive solutions to the Yang-Baxter Equation. These structures occupy nowadays a central role in the classification of set-theoretical solutions.  Thus, based on the fact that Hopf algebras are the quantum analogues of groups, I. Angiono, C. Galindo and L. Vendramin \cite{AGV} provide the notion of a Hopf brace as the quantum version of a skew brace, that is, passing from a group-theoretic structure to one where the underlying structures are vector spaces. More concretely, if $(H,\epsilon,\Delta)$ is a coalgebra, a Hopf brace structure $\mathbb{H} = (H_1, H_2)$ on $H$ consists of a pair of Hopf algebras  \[H_{1}=(H,1,\cdot,\epsilon,\Delta,S) \textnormal{ and } H_{2}=(H,1_{\circ},\circ,\epsilon,\Delta,T),\] 
such that the following compatibility condition between the products holds:
\[g\circ(h\cdot k)=(g_{1}\circ h)\cdot S(g_{2})\cdot(g_{3}\circ k)\]
for all $g,h,k\in H$. As established by the authors, cocommutative Hopf braces also yield solutions to the QYBE, and such solutions are involutive in the case that $H_{1}$ is a commutative Hopf algebra \cite[Corollaries 2.4 and~2.5]{AGV}.

Once the notion of brace in the abelian setting was established, Rump  gave a definition of module over a brace 
$(A, +, \circ)$ \cite[Definition 3]{Rump}. A remarkable point in this definition is that it only relies on a representation of the second group $(A, \circ)$, due to the fact of considering abelian groups. When the notion of brace is extended to non-abelian contexts, Rump's first approach needs to be extended to include not only the second structure but also the first one. In this line,  H. Zhu \cite{ZHU} gives a quite restrictive notion of modules over  a Hopf brace. On the other hand, R. González \cite{RGON} obtains a wider definition that contains Zhu's one, and they happen to be equivalent if the underlying coalgebra is cocommutative. However, if we do not ask for any extra condition, both definitions differ. For example, while González's one includes $\mathbb H$ as a $\mathbb H$-module, Zhu's one does not, so González's approach seems to be more suitable in non-cocommutative contexts.

At this point, modules over a Hopf  brace need to involve (at least) two structures of modules over two different (although related) Hopf algebras. However, Y. Kozakai and C. Tsang  point out in \cite[Remarks 2.2 and 3.2]{KOZTSANG} that modules in the sense of Zhu can be regarded as modules over an algebra, indeed, as the semidirect product of the underline algebras of the Hopf brace. However, they do not obtain an equivalence between both categories of modules. 

Thus, in this work we deal with two clearly distinct questions. The first one is related to the possibility of interpreting the category of modules over a Hopf brace as a category of modules over an algebra. In Section \ref{sec_equivalence}, using the definition proposed by González, we prove that the category of modules over a Hopf brace $\mathbb{H}$ is isomorphic to the category of modules over a smash product algebra arising from the underlying Hopf brace $\mathbb{H}$ (see Theorem \ref{mainth}).  Also, as a consequence of this theorem, we can assure that the category of modules over a Hopf brace  in González's sense can be seen as a category of Doi-Hopf modules (see Corollary \ref{mainthcor}). 

The second question is addressed in Section \ref{comparison}  and it concerns the relation between González's and Zhu's modules. In Theorem \ref{firstM} we characterize Zhu's modules in terms of the so-called cocommutativity class (see \cite{AVG}). This characterization may help in skipping long calculations when Zhu's approach could be needed.

\section{Preliminaries about Hopf algebras and Hopf braces}
We work throughout this article in a strict symmetric monoidal framework. By Mac Lane's coherence theorem, the strictness assumption entails no loss of generality, since any monoidal category is monoidal equivalent to a strict one. Therefore, the statements proved below remain valid in the general non-strict setting, up to the standard modifications. In what follows, ${\sf C}$ denotes a strict symmetric monoidal category with tensor product $\otimes$, unit object $K$, and symmetry $c$. For any morphism $f\colon M\rightarrow N$ and any object $P$ in ${\sf C}$, we shall use the notation $P\otimes f$ and $f\otimes P$ instead of ${\rm id}_{P}\otimes f$ and $f\otimes{\rm id}_{P}$, respectively. For the reader's convenience, in the first part of this section we recall some standard definitions that we will use repeatedly in the development of the article.

\begin{definition}
An algebra in {\sf C} is a triple $A=(A,\eta_{A},\mu_{A})$ where $\eta_{A}\colon K\rightarrow A$ (unit) and $\mu_{A}\colon A\otimes A\rightarrow A$ (product) are morphisms in {\sf C} satisfying the following conditions:
\begin{gather*}
\mu_{A}\circ (\eta_{A}\otimes A)={\rm id}_{A}=\mu_{A}\circ(A\otimes\eta_{A}),\quad\mu_{A}\circ(\mu_{A}\otimes A)=\mu_{A}\circ(A\otimes\mu_{A}).
\end{gather*}
	
Note that, given $B=(B,\eta_{B},\mu_{B})$ another algebra in {\sf C}, the tensor product $A\otimes B$ admits an algebra structure whose unit and product are defined as follows:
$$
\eta_{A\otimes B}\coloneqq \eta_{A}\otimes\eta_{B}, \quad
\mu_{A\otimes B}\coloneqq (\mu_{A}\otimes\mu_{B})\circ(A\otimes c_{B,A}\otimes B).
$$
	
Moreover, a morphism $f\colon A\rightarrow B$ in ${\sf C}$ is said to be an algebra morphism if it preserves the unit and the product, that is,
$$
f\circ\eta_{A}=\eta_{B}, \quad f\circ\mu_{A}=\mu_{B}\circ(f\otimes f).
$$
	
Dually, we will say that a triple $C=(C,\varepsilon_{C},\delta_{C})$ is a coalgebra in {\sf C} if  $\varepsilon_{C}\colon C\rightarrow K$ (counit) and $\delta_{C}\colon C\rightarrow C\otimes C$ (coproduct) are morphisms in {\sf C} satisfying the following conditions:
\begin{gather*}
(\varepsilon_{C}\otimes C)\circ\delta_{C}={\rm id}_{C}=(C\otimes\varepsilon_{C})\circ\delta_{C},\quad(\delta_{C}\otimes C)\circ\delta_{C}=(C\otimes\delta_{C})\circ\delta_{C}.
\end{gather*}
	
Likewise, if $D=(D,\varepsilon_{D},\delta_{D})$ is another coalgebra in {\sf C}, the tensor product $C\otimes D$ is a coalgebra with the following counit and coproduct:
$$
\varepsilon_{C\otimes D}\coloneqq\varepsilon_{C}\otimes\varepsilon_{D},\quad
\delta_{C\otimes D}\coloneqq(C\otimes c_{C,D}\otimes D)\circ(\delta_{C}\otimes\delta_{D}).
$$
	
Besides, a morphism $g\colon C\rightarrow D$ in ${\sf C}$ is a coalgebra morphism if the equalities
$$
\varepsilon_{D}\circ g=\varepsilon_{C}, \quad\delta_{D}\circ g=(g\otimes g)\circ\delta_{C}
$$
hold, i.e., $g$ preserves the counit and it is comultiplicative.
\end{definition}

\begin{definition}
Let $C$ be a coalgebra and let $A$ be an algebra in ${\sf C}$. Let $f,g:C\rightarrow A$ be two morphisms in ${\sf C}$. We define the convolution product of $f$ and $g$ as 
\[f\ast g\coloneqq\mu_{A}\circ (f\otimes g)\circ\delta_{C}.\]

The set of morphisms in ${\sf C}$ from $C$ to $A$, denoted by $\operatorname{Hom}_{{\sf C}}(C,A)$, is a monoid with the convolution product whose unit is given by $\eta_{A}\circ\varepsilon_{C}=\varepsilon_{C}\otimes\eta_{A}$.
\end{definition}

\begin{definition}
Let $A=(A,\eta_{A},\mu_{A})$ be an algebra in ${\sf C}$. A left $A$-module in ${\sf C}$ is a pair $(M,\varphi_{M})$ where $\varphi_{M}\colon A\otimes M\rightarrow M$ (action) is a morphism in ${\sf C}$ which satisfies the following equalities:
\begin{equation}\label{actioneta}
\varphi_{M}\circ(\eta_{A}\otimes M)={\rm id}_{M},
\end{equation}
\begin{equation}
\label{actionprod}
\varphi_{M}\circ(A\otimes\varphi_{M})=\varphi_{M}\circ(\mu_{A}\otimes M).
\end{equation}

Let $(N,\varphi_{N})$ be a left $A$-module. A morphism $f\colon M\rightarrow N$ is a morphism of left $A$-modules (or, in other words, $f$ is $A$-linear) if 
\begin{gather}\label{mod_mor}
	f\circ\varphi_{M}=\varphi_{N}\circ(A\otimes f).
\end{gather}

With the previous morphisms, these objects constitute a category denoted by ${}_{A}{\sf Mod}$.

In a dual way, it is possible to define the category of right $C$-comodules for a coalgebra $C=(C,\varepsilon_{C},\delta_{C})$. A right $C$-comodule in ${\sf C}$ is a pair  $(M,\rho_{M})$ where $\rho_{M}\colon M\rightarrow M\otimes C$ (coaction) is a morphism in ${\sf C}$ which satisfies the following equalities:
\begin{equation*}\label{coactionvar}
(M\otimes \varepsilon_{C})\circ \rho_{M}={\rm id}_{M},
\end{equation*}
\begin{equation*}
\label{coactioncoprod}
(\rho_{M}\otimes C)\circ \rho_{M}=(M\otimes \delta_{C})\circ \rho_{M}.
\end{equation*}

Let $(N,\rho_{N})$ be a right $C$-comodule. A morphism $f\colon M\rightarrow N$ is a morphism of right  $C$-comodules (or, in other words, $f$ is $C$-colinear) if 
\begin{gather*}\label{comod_mor}
\rho_{N}\circ f=(f\otimes C)\circ \rho_{M}.
\end{gather*}

With the previous morphisms, these objects constitute a category denoted by ${\sf Mod}^C$.
\end{definition}

When an object in the category is endowed with both an algebra and a coalgebra structure that are compatible with each other, we obtain the notion of a bialgebra.
\begin{definition}
	A 5-tuple $B=(B,\eta_{B},\mu_{B},\varepsilon_{B},\delta_{B})$ is said to be a bialgebra in {\sf C} if $(B,\eta_{B},\mu_{B})$ is an algebra in {\sf C}, $(B,\varepsilon_{B},\delta_{B})$ is a coalgebra in {\sf C} and $\eta_{B}$ and $\mu_{B}$ are coalgebra morphisms (equivalently, if $\varepsilon_{B}$ and $\delta_{B}$ are algebra morphisms).
	
	If $D=(D,\eta_{D},\mu_{D},\varepsilon_{D},\delta_{D})$ is another bialgebra in ${\sf C}$, a morphism in ${\sf C}$ $f\colon B\rightarrow D$ is a bialgebra morphism if it is simultaneously an algebra and a coalgebra morphism.
\end{definition}

In the following definition we recall the notion of Hopf algebra.
\begin{definition}
	A 6-tuple $H=(H,\eta_{H},\mu_{H},\varepsilon_{H},\delta_{H},\lambda_{H})$ is said to be a Hopf algebra in {\sf C} when $(H,\eta_{H},\mu_{H},\varepsilon_{H},\delta_{H})$ is a bialgebra such that there exists an endomorphism $\lambda_{H}\colon H\rightarrow H$, called the antipode, satisfying the following identity:
	\begin{gather}\label{antipode}
		\lambda_{H}\ast {\rm id}_{H}=\varepsilon_{H}\otimes\eta_{H}={\rm id}_{H}\ast \lambda_{H}.
\end{gather}
\end{definition}

The equation \eqref{antipode} means that the antipode, $\lambda_{H}$, is the inverse of the identity for the convolution product in $\operatorname{Hom}_{{\sf C}}(H,H)$ and, as a consequence, it is unique.

A Hopf algebra $H$ is said to be commutative when the underlying algebra is commutative, i.e., if the equality $\mu_{H}\circ c_{H,H}=\mu_{H}$, whereas $H$ is said to be cocommutative when the underlying coalgebra is cocommutative, that is, in case that $c_{H,H}\circ\delta_{H}=\delta_{H}$.

The antipode of a Hopf algebra $H$ is antimultiplicative and anticomultiplicative, that is, it satisfies:
\begin{equation*}\label{a-antip1}
	\lambda_{H}\circ\mu_{H}=\mu_{H}\circ c_{H,H}\circ(\lambda_{H}\otimes\lambda_{H}),
\end{equation*}
\begin{equation}
\label{a-antip2}
	\delta_{H}\circ\lambda_{H}=(\lambda_{H}\otimes\lambda_{H})\circ c_{H,H}\circ \delta_{H}{\color{blue} .}
\end{equation}

Moreover the antipode preserves the unit and the counit, i.e.,
$\lambda_{H}\circ\eta_{H}=\eta_{H}$,
$\varepsilon_{H}\circ\lambda_{H}=\varepsilon_{H}.$

As a consequence, if $H$ is commutative, then  $\lambda_H$ is an algebra morphism and, if $H$ is cocommutative,  $\lambda_H$ is a coalgebra morphism. In any of these cases the equality  
\begin{gather}\label{lambdasquareid}
	\lambda_{H}\circ\lambda_{H}={\rm id}_{H}
\end{gather}
holds,  i.e., $\lambda_H$ is involutive.

Given $B=(B,\eta_{B},\mu_{B},\varepsilon_{B},\delta_{B},\lambda_{B})$ another Hopf algebra in ${\sf C}$, a morphism $f\colon H\rightarrow B$ is a Hopf algebra morphism if it is a bialgebra morphism. Note that, in this case, $f$ commutes with the antipodes, that is,
\begin{gather*}\label{morant}
	\lambda_{B}\circ f=f\circ\lambda_{H}.
\end{gather*}

\begin{definition}
	An object $P$ in {\sf C} is profinite if there exists an object $P^{\ast}$ in ${\sf C}$, called the dual of $P$, and a  {\sf C}-adjuntion $-\otimes P\dashv -\otimes P^{\ast}$. We will denote by $a_{P}$ and $b_{P}$ the unit and the counit of the previous  {\sf C}-adjuntion, respectively. Recall that the unit and the counit are collections of natural isomorphisms in {\sf C} given by $$a_{P}=\{a_{P}(X):X\rightarrow X\otimes P\otimes P^{\ast}; \; X\in {\sf C}\},\quad b_{P}=\{b_{P}(X):X\otimes  P^{\ast}\otimes P\rightarrow X; \; X\in {\sf C}\}$$ satisfying the triangular identities.

\end{definition}

Let $H$ be a Hopf algebra in {\sf C}. If $H$ is profinite, then its dual $H^{\ast}$ is a Hopf algebra in ${\sf C}$ where 
\begin{gather*}\eta_{H^{\ast}}=(\varepsilon_{H}\otimes H^{\ast})\circ a_{H}(K), \quad \varepsilon_{H^{\ast}}=b_{H}(K)\circ (H^{\ast}\otimes \eta_{H}),\\
\mu_{H^{\ast}}=((b_{H}(K)\circ (H^{\ast}\otimes b_{H}(K) \otimes H)) \otimes H^{\ast})\circ (H^{\ast}\otimes H^{\ast}\otimes ((c_{H,H}\circ \delta_{H})\otimes H^{\ast})\circ a_{H}(K)),\\
\delta_{H^{\ast}}=((b_{H}(K)\circ (H^{\ast}\otimes (\mu_{H}\circ c_{H,H})))\otimes H^{\ast}\otimes H^{\ast})\circ (H^{\ast}\otimes ((H\otimes a_{H}(K)\otimes H^{\ast})\circ a_{H}(K))),\\
\lambda_{H^{\ast}}=((b_{H}(K)\circ (H^{\ast}\otimes\lambda_{H} ))\otimes H^{\ast})\circ (H^{\ast}\otimes a_{H}(K)).
\end{gather*}

\begin{definition}
Let $H$ be a Hopf algebra and $A$ an algebra in ${\sf C}$. A left $H$-module $(A,\varphi_{A})$ is said to be a left $H$-module algebra if $\eta_{A}$ and $\mu_{A}$ are $H$-linear morphisms, that is, if the following equalities
\begin{equation*}
\label{Hmodalg1}\varphi_{A}\circ(H\otimes\eta_{A})=\varepsilon_{H}\otimes\eta_{A},
\end{equation*}
\begin{equation*}		
\label{Hmodalg2}\varphi_{A}\circ(H\otimes\mu_{A})=\mu_{A}\circ\varphi_{A\otimes A},
\end{equation*} 
hold, where $\varphi_{A\otimes A}\coloneqq(\varphi_{A}\otimes\varphi_{A})\circ(H\otimes c_{H,A}\otimes A)\circ(\delta_{H}\otimes A\otimes A)$.
\end{definition}

\begin{example}
Given a Hopf algebra $H$, $(H,\varphi_{H}^{{\rm ad}})$ is a left $H$-module algebra where $\varphi_{H}^{{\rm ad}}$ is the so-called adjoint action defined by 
\[\varphi_{H}^{{\rm ad}}\coloneqq \mu_{H}\circ(\mu_{H}\otimes\lambda_{H})\circ(H\otimes c_{H,H})\circ(\delta_{H}\otimes H).\]
\end{example}

The next theorem is a very well known result in the theory of Hopf algebras, and provides a generalization of the semidirect product of groups, as we will make explicit later:

\begin{theorem}\label{th-smash}

Let $H$ be a Hopf algebra and $(A, \varphi_A)$ a left $H$-module algebra, and define
\[A\sharp H=(A\otimes H,\eta_{A\sharp H}\coloneqq\eta_{A}\otimes\eta_{H},\mu_{A\sharp H}\coloneqq(\mu_{A}\otimes\mu_{H})\circ (A\otimes \Psi^{H}_{A}\otimes H)),\]
where $\Psi^{H}_{A}\coloneqq (\varphi_{A}\otimes H)\circ(H\otimes c_{H,A})\circ(\delta_{H}\otimes A)$.
Then, $A\sharp H$ is an algebra in ${\sf C}$, known as the smash product algebra. 
\end{theorem}

\begin{definition}
Let $H$ be a Hopf algebra and $C$ a coalgebra in ${\sf C}$. A left $H$-module $(C,\varphi_{C})$ satisfying that $\varepsilon_{C}$ and $\delta_{C}$ are $H$-linear morphisms is said to be a left $H$-module coalgebra. In other words, $(C,\varphi_{C})$ is a left $H$-module coalgebra if the equalities 
\begin{equation}
\label{Hmodcoalg1}\varepsilon_{C}\circ\varphi_{C}=\varepsilon_{H}\otimes\varepsilon_{C},
\end{equation}
\begin{equation}\label{Hmodcoalg2}\delta_{C}\circ\varphi_{C}=\varphi_{C\otimes C}\circ(H\otimes\delta_{C})
\end{equation} 
hold. Note that \eqref{Hmodcoalg1} and \eqref{Hmodcoalg2} are equivalent to the fact that $\varphi_{C}$ is a coalgebra morphism.
\end{definition}

Let $H$ and $B$ be Hopf algebras in ${\sf C}$. If $(B,\varphi_{B})$ is both a left $H$-module algebra and a left $H$-module coalgebra, then it will be called a left $H$-module algebra-coalgebra.

A fundamental concept for this article is the notion of cocommutativity class, introduced in \cite{AVG} by Alonso Álvarez, Fernández Vilaboa and González Rodríguez. Informally, it is a condition given for a module which is weaker than the usual cocommutativity.
\begin{definition}
	Let $H$ be a Hopf algebra in ${\sf C}$. A left $H$-module $(M,\varphi_{M})$ belongs to the cocommutativity class of $H$ if it satisfies the following equality:
	\begin{gather}\label{ccclass}
		(\varphi_{M}\otimes H)\circ(H\otimes c_{H,M})\circ(\delta_{H}\otimes M)=(\varphi_{M}\otimes H)\circ(H\otimes c_{H,M})\circ((c_{H,H}\circ\delta_{H})\otimes M).
	\end{gather}
	
	Note that, thanks to the symmetric character of ${\sf C}$, \eqref{ccclass} is equivalent to 
	\begin{gather}\label{ccclass2}
		(H\otimes\varphi_{M})\circ(\delta_{H}\otimes M)=(H\otimes\varphi_{M})\circ((c_{H,H}\circ\delta_{H})\otimes M).
	\end{gather}
	
	Obviously, in the case that $H$ is a cocommutative Hopf algebra, \eqref{ccclass} automatically holds.
\end{definition}

This condition arises naturally in a variety of contexts as will be shown in the following examples.
\begin{example}\label{smashHA}
	The cocommutativity class turns out to be a necessary condition for the smash product algebra to be a Hopf algebra. More precisely, let $A$ and $H$ be Hopf algebras such that $(A,\varphi_{A})$ is a left $H$-module algebra. If, in addition, $(A,\varphi_{A})$ is a left $H$-module coalgebra satisfying \eqref{ccclass}, then 
	\[A\sharp H=(A\otimes H,\eta_{A}\otimes\eta_{H},\mu_{A\sharp H},\varepsilon_{A}\otimes\varepsilon_{H},\delta_{A\otimes H},\lambda_{A\sharp H})\]
	is a Hopf algebra where $\lambda_{A\sharp H}\coloneqq \Psi^{H}_{A}\circ(\lambda_{H}\otimes\lambda_{A})\circ c_{A,H}.$
	
	Note that, when ${\sf C}={\sf Set}$, any Hopf algebra is a group. Consequently, in this setting, the smash product Hopf algebra $\sharp$ coincides with the well-known semidirect product group $\ltimes$.
\end{example} 
\begin{example}\label{ccclass_adjoint}
	Let $H$ be a Hopf algebra whose antipode $\lambda_{H}$ is an isomorphism in ${\sf C}$. The cocommutativity class also characterizes when the adjoint action of $H$ is a coalgebra morphism, i.e., $\varphi_{H}^{{\rm ad}}$ is a coalgebra morphism if and only if $(H,\varphi_{H}^{{\rm ad}})$ belongs to the cocommutativity class of $H$. 
	
Let us see the proof of the equivalence. At first, by the counit properties, note 
that $\varepsilon_{H}\circ \varphi_{H}^{{\rm ad}}=\varepsilon_{H}\otimes \varepsilon_{H}.$ Secondly, for the adjoint action the following identity 
\begin{equation}\label{z1}
\delta_{H}\circ \varphi_{H}^{{\rm ad}}=(\mu_{H}\otimes H)\circ(\mu_{H}\otimes ((\lambda_{H}\otimes\varphi_{H}^{{\rm ad}})\circ((c_{H,H}\circ\delta_{H})\otimes H)))\circ(H\otimes c_{H,H}\otimes H)\circ(\delta_{H}\otimes\delta_{H})
\end{equation}
holds. Indeed,
\begin{itemize}
\itemindent=-32pt 
\item[]$\hspace{0.38cm}\delta_{H}\circ \varphi_{H}^{{\rm ad}}$
\item[]$=\delta_{H}\circ\mu_{H}\circ(\mu_{H}\otimes\lambda_{H})\circ(H\otimes c_{H,H})\circ(\delta_{H}\otimes H)$
\item[]$=(\mu_{H}\otimes\mu_{H})\circ(H\otimes c_{H,H}\otimes H)\circ(((\mu_{H}\otimes\mu_{H})\circ(H\otimes c_{H,H}\otimes H)\circ(\delta_{H}\otimes\delta_{H}))\otimes(\delta_{H}\circ\lambda_{H}))$
\item[]$\hspace{0.38cm}\circ(H\otimes c_{H,H})\circ(\delta_{H}\otimes H)$ {\footnotesize (by the condition of coalgebra morphism for $\mu_{H}$)}
\item[]$=(\mu_{H}\otimes\mu_{H})\circ(H\otimes c_{H,H}\otimes H)\circ(((\mu_{H}\otimes\mu_{H})\circ(H\otimes c_{H,H}\otimes H)\circ(\delta_{H}\otimes\delta_{H}))\otimes((\lambda_{H}\otimes\lambda_{H})$
\item[]$\hspace{0.38cm}\circ c_{H,H}\circ\delta_{H}))\circ(H\otimes c_{H,H})\circ(\delta_{H}\otimes H)$ {\footnotesize (by \eqref{a-antip2})}
\item[]$=((\mu_{H}\circ (\mu_{H}\otimes\lambda_{H})\circ(H\otimes c_{H,H}))\otimes(\mu_{H}\circ (\mu_{H}\otimes\lambda_{H})\circ(H\otimes c_{H,H})))$
\item[]$\hspace{0.38cm}\circ(H\otimes((H\otimes c_{H,H}\otimes H)\circ (c_{H,H}\otimes c_{H,H})\circ (H\otimes c_{H,H}\otimes H)\circ(\delta_{H}\otimes H\otimes H))\otimes H)$
\item[]$\hspace{0.38cm}\circ(((H\otimes\delta_{H})\circ\delta_{H})\otimes\delta_{H})$ {\footnotesize (by naturality of $c$, the symmetric character of ${\sf C}$ and coassociativity of $\delta_{H}$)}
\item[]$=((\mu_{H}\circ(\mu_{H}\otimes\lambda_{H}))\otimes\varphi_{H}^{{\rm ad}})\circ(H\otimes((c_{H,H}\otimes H)\circ (H\otimes c_{H,H})\circ ((c_{H,H}\circ\delta_{H})\otimes H))\otimes H)$
\item[]$\hspace{0.38cm}\circ(\delta_{H}\otimes\delta_{H})$ {\footnotesize (by naturality of $c$)}
\item[]$=(\mu_{H}\otimes H)\circ(\mu_{H}\otimes ((\lambda_{H}\otimes\varphi_{H}^{{\rm ad}})\circ((c_{H,H}\circ\delta_{H})\otimes H)))\circ(H\otimes c_{H,H}\otimes H)\circ(\delta_{H}\otimes\delta_{H})$ 
\item[]$\hspace{0.38cm}${\footnotesize (by naturality of $c$)}. 
\end{itemize}
	
Therefore, on the one hand, let us assume that \eqref{ccclass2} holds for $(H,\varphi_{H}^{{\rm ad}})$. Then,
\begin{itemize}
\itemindent=-32pt 
\item[]$\hspace{0.38cm}\delta_{H}\circ \varphi_{H}^{{\rm ad}}$
\item[]$=(\mu_{H}\otimes H)\circ(\mu_{H}\otimes ((\lambda_{H}\otimes\varphi_{H}^{{\rm ad}})\circ((c_{H,H}\circ\delta_{H})\otimes H)))\circ(H\otimes c_{H,H}\otimes H)\circ(\delta_{H}\otimes\delta_{H})$ {\footnotesize (by \eqref{z1})}
\item[]$=(\mu_{H}\otimes H)\circ(\mu_{H}\otimes ((\lambda_{H}\otimes\varphi_{H}^{{\rm ad}})\circ(\delta_{H}\otimes H)))\circ(H\otimes c_{H,H}\otimes H)\circ(\delta_{H}\otimes\delta_{H})$ {\footnotesize (by \eqref{ccclass2})}
\item[]$=(\varphi_{H}^{{\rm ad}}\otimes \varphi_{H}^{{\rm ad}})\circ(H\otimes c_{H,H}\otimes H)\circ(\delta_{H}\otimes\delta_{H})$ {\footnotesize (by naturality of $c$ and coassociativity of $\delta_{H}$)}, 
\end{itemize}
and, as a consequence, $\varphi_{H}^{{\rm ad}}$ is a coalgebra morphism.
	
On the other hand, let us assume that $\varphi_{H}^{{\rm ad}}$ is a coalgebra morphism. Consequently, by \eqref{z1},
\begin{itemize}
\itemindent=-32pt 
\item[]$\hspace{0.38cm}(\mu_{H}\otimes H)\circ(\mu_{H}\otimes ((\lambda_{H}\otimes\varphi_{H}^{{\rm ad}})\circ((c_{H,H}\circ\delta_{H})\otimes H)))\circ(H\otimes c_{H,H}\otimes H)\circ(\delta_{H}\otimes\delta_{H})$
\item[]$=\delta_{H}\circ \varphi_{H}^{{\rm ad}}$ {\footnotesize (by \eqref{z1})}
\item[]$=(\varphi_{H}^{{\rm ad}}\otimes\varphi_{H}^{{\rm ad}})\circ(H\otimes c_{H,H}\otimes H)\circ(\delta_{H}\otimes\delta_{H})$ {\footnotesize (by the condition of coalgebra morphism for $\varphi_{H}^{{\rm ad}}$)}
\item[]$=(\mu_{H}\otimes H)\circ(\mu_{H}\otimes((\lambda_{H}\otimes\varphi_{H}^{{\rm ad}})\circ(\delta_{H}\otimes H)))\circ(H\otimes c_{H,H}\otimes H)\circ(\delta_{H}\otimes \delta_{H})$ {\footnotesize (by coassociativity of }
\item[]$\hspace{0.38cm}${\footnotesize  $\delta_{H}$ and naturality of $c$)}. 
\end{itemize}
	
By composing the previous equality on the right with $((\lambda_{H}\otimes H)\circ\delta_{H})\otimes H$, and on the left with $\mu_{H}\otimes H$, we obtain that
\begin{align}\label{z2}
\begin{split}
&(\mu_{H}\otimes H)\circ(H\otimes((\lambda_{H}\otimes\varphi_{H}^{{\rm ad}})\circ((c_{H,H}\circ\delta_{H})\otimes H)))\circ (c_{H,H}\otimes H)\circ(H\otimes \delta_{H})\\=\,&(\mu_{H}\otimes H)\circ(H\otimes((\lambda_{H}\otimes\varphi_{H}^{{\rm ad}})\circ(\delta_{H}\otimes H)))\circ (c_{H,H}\otimes H)\circ(H\otimes \delta_{H}).
\end{split}
\end{align}
	
Besides, by composing on the right with $c_{H,H}$ and then using the naturality of $c$ and the symmetric character of ${\sf C}$, \eqref{z2} is equivalent to 
\begin{align}\label{z3}
\begin{split}
&(\mu_{H}\otimes H)\circ(H\otimes ((\lambda_{H}\otimes\varphi_{H}^{{\rm ad}})\circ ((c_{H,H}\circ\delta_{H})\otimes H)\circ c_{H,H}))\circ (\delta_{H}\otimes H)\\=\,&(\mu_{H}\otimes H)\circ(H\otimes ((\lambda_{H}\otimes\varphi_{H}^{{\rm ad}})\circ (\delta_{H}\otimes H)\circ c_{H,H}))\circ (\delta_{H}\otimes H)
\end{split}
\end{align}
\noindent since $c_{H,H}$ is an isomorphism in ${\sf C}$.
	
As a result, if we compose the equality \eqref{z3} on the right with $((\lambda_{H}\otimes H)\circ\delta_{H})\otimes H$, and on the left with $\mu_{H}\otimes H$, we conclude that 
\begin{gather*}
(\lambda_{H}\otimes\varphi_{H}^{{\rm ad}})\circ ((c_{H,H}\circ\delta_{H})\otimes H)\circ c_{H,H}=(\lambda_{H}\otimes\varphi_{H}^{{\rm ad}})\circ (\delta_{H}\otimes H)\circ c_{H,H},
\end{gather*}
which is equivalent to \eqref{ccclass2} for $(H,\varphi_{H}^{{\rm ad}})$ because $c_{H,H}$ and $\lambda_{H}$ are isomorphisms.

Observe that, in order to show that $\varphi_{H}^{{\rm ad}}$ is a coalgebra morphism if $(H,\varphi_{H}^{{\rm ad}})$ is in the cocommutativity class of $H$, it is not necessary to assume that the antipode of $H$ is an isomorphism. However, this assumption is not overly restrictive, since it holds when $H$ is profinite \cite[Corollary 1]{PuraE} or whenever the Hopf algebra $H$ is commutative or cocommutative thanks to \eqref{lambdasquareid}. 
\end{example}

\begin{definition}
	Let $H$ be a Hopf algebra and $A$ an algebra in ${\sf C}$. A right $H$-comodule $(A,\rho_{A})$ is said to be a right $H$-comodule algebra if $\eta_{A}$ and $\mu_{A}$ are $H$-colinear morphisms, that is, if the following equalities 
\begin{gather}
\label{Hcomodalg1}\rho_{A}\circ \eta_{A}=\eta_{A}\otimes \eta_{H},\\
\label{Hcomodalg2} \rho_{A}\circ \mu_{A}=(\mu_{A}\otimes H)\circ \rho_{A\otimes A},
\end{gather} 
hold, where $ \rho_{A\otimes A}\coloneqq(A\otimes A\otimes \mu_{H})\circ (A\otimes c_{H,A}\otimes H)\circ (\rho_{A}\otimes \rho_{A})$. One shows easily that equations \eqref{Hcomodalg1} and \eqref{Hcomodalg2} hold if and only if $\rho_{A}$ is an algebra morphism.
\end{definition}

\begin{definition}
Let $H$ be a Hopf algebra and let  $(A,\rho_{A})$  be a  right $H$-comodule  algebra. We define the category of left-right $(A,H)$-Doi-Hopf modules as the one whose objects are triples $(M, \varphi_{M}, \rho_{M})$ such that $(M, \varphi_{M})$ is a left $A$-module,  $(M, \rho_{M})$ is a right $H$-comodule and the following identity holds:
\begin{equation*}
\label{lrHopfmod}\rho_{M}\circ \varphi_{M}=(\varphi_{M}\otimes \mu_{H})\circ (A\otimes c_{H,M}\otimes H)\circ (\rho_{A}\otimes \rho_{M}).
\end{equation*} 

The morphisms in this category are left $A$-linear and right $H$-colinear morphisms. In what follows we will denote it by ${}_{A}{\sf M}^{H}$.
\end{definition}

If $H$ is a profinite Hopf algebra and $(A, \varphi_{A})$ is a left $H$-module algebra in ${\sf C}$, then $A$ it is also a right $H^{\ast}$-comodule algebra via 
$$\widehat{\rho}_{A}\coloneqq(\varphi_{A}\otimes H^{\ast})\circ (H\otimes c_{H^{\ast},A})\circ (a_{H}(K)\otimes A).$$

Dually, if   $(A, \rho_{A})$ is a  right $H^{\ast}$-comodule algebra, then $A$ is a  left $H$-module algebra with the action 
$$\widetilde{\varphi}_{A}\coloneqq(A\otimes b_{H}(K))\circ (\rho_{A}\otimes H)\circ c_{H,A}.$$

Therefore, if $(A, \varphi_{A})$ is a left $H$-module algebra, then we have a functor 
$${\sf S}\colon {}_{A \sharp H}{\sf Mod}\longrightarrow {}_{A}{\sf M}^{H^{\ast}}$$
defined on objects by 
$${\sf S}((M, \psi_{M}))\coloneqq(M, \varphi_{M}^{\ast}, \rho_{M}^{\ast}),$$
where $\varphi_{M}^{\ast}$, $\rho_{M}^{\ast}$ are 
\begin{gather*}
	\varphi_{M}^{\ast}\coloneqq\psi_{M}\circ (A\otimes \eta_{H}\otimes M),\\
	\rho_{M}^{\ast}\coloneqq(\psi_{M}\otimes H^{\ast})\circ (A\otimes H\otimes c_{H^{\ast},M} )\circ (\eta_{A}\otimes a_{H}(K)\otimes M),
\end{gather*}
and by the identity on morphisms. Also, there exists  another functor 
$${\sf R}\colon {}_{A}{\sf M}^{H^{\ast}}  \longrightarrow {}_{A\sharp H}{\sf Mod}$$
defined on objects by 
$${\sf R}((M, \varphi_{M},\rho_{M}))\coloneqq(M, \psi_{M}^{\sharp}\coloneqq(\varphi_{M}\otimes b_{H}(K))\circ (A\otimes ((\rho_{M}\otimes H)\circ c_{H,M}))),$$
and by the identity on morphisms.  These functors are mutually inverse, leading to the following theorem: 

\begin{theorem}
	\label{smash} Let $H$ be a profinite Hopf algebra in {\sf C}. If $(A, \varphi_{A})$ is  a left $H$-module algebra, then the categories 
	${}_{A}{\sf M}^{H^{\ast}}$ and ${}_{A \sharp H}{\sf Mod}$ are isomorphic. 
\end{theorem}

The previous theorem is a monoidal version, with different sides in the category of Hopf modules, of a result proved by Fischman in her doctoral thesis, whose proof can be found in \cite[Proposition 1.6]{CFM}.

We now introduce the notion of a Hopf brace in a general monoidal framework. These objects have been introduced by Angiono, Galindo and Vendramin \cite{AGV} in the vector space setting as the quantum version of a skew brace \cite{GV}.
\begin{definition}
	Let $(H,\varepsilon_{H},\delta_{H})$ be a coalgebra in {\sf C} and let us assume that $H$ admits two different algebra structures in {\sf C}: $(H,\eta_{H}^{1},\mu_{H}^{1})$ and $(H,\eta_{H}^{2},\mu_{H}^{2})$. We will say that a 9-tuple
	\[(H,\eta_{H}^{1},\mu_{H}^{1},\eta_{H}^{2},\mu_{H}^{2},\varepsilon_{H},\delta_{H},\lambda_{H}^{1},\lambda_{H}^{2})\]
	is a Hopf brace in {\sf C} if the following requirements hold:
	\begin{itemize}
		\item[(i)] $H_{1}=(H,\eta_{H}^{1},\mu_{H}^{1},\varepsilon_{H},\delta_{H},\lambda_{H}^{1})$ is a Hopf algebra in {\sf C}.
		\item[(ii)]  $H_{2}=(H,\eta_{H}^{2},\mu_{H}^{2},\varepsilon_{H},\delta_{H},\lambda_{H}^{2})$ is a Hopf algebra in {\sf C}.
		\item[(iii)] The following identity involving the products $\mu_{H}^{1}$ and $\mu_{H}^{2}$ holds:
		\begin{gather}\label{compatHbrace}
			\mu_{H}^{2}\circ(H\otimes\mu_{H}^{1})=\mu_{H}^{1}\circ(\mu_{H}^{2}\otimes\Gamma_{H_{1}})\circ(H\otimes c_{H,H}\otimes H)\circ(\delta_{H}\otimes H\otimes H),
		\end{gather}
		where 
		\begin{equation}\label{def_GammaH1}
			\Gamma_{H_{1}}\coloneqq \mu_{H}^{1}\circ(\lambda_{H}^{1}\otimes\mu_{H}^{2})\circ(\delta_{H}\otimes H).
		\end{equation}
	\end{itemize}
	
	Following the notation introduced in \cite{RGON}, we will denote Hopf braces by $\mathbb{H}=(H_{1},H_{2})$ or, when there is no confusion between the Hopf algebras involved, only by $\mathbb{H}$. A Hopf brace $\mathbb{H}$ is said to be cocommutative if $c_{H,H}\circ\delta_{H}=\delta_{H}$.
	
	These structures constitute a category whose morphisms are defined as follows: Given another Hopf brace $\mathbb{B}=(B_{1},B_{2})$, a morphism $f\colon \mathbb{H}\rightarrow \mathbb{B}$ is a morphism of Hopf braces if $f\colon H_{1}\rightarrow B_{1}$ and $f\colon H_{2}\rightarrow B_{2}$ are both Hopf algebra morphisms in ${\sf C}$. The category of Hopf braces will be denoted by ${\sf HBr}$. Cocommutative Hopf braces form a full subcategory of ${\sf HBr}$ which we will denote by ${\sf cocHBr}$.
\end{definition}
\begin{example}[Skew braces]\label{skewbraces}
	Since a Hopf algebra in the symmetric monoidal category of sets, {\sf Set}, is just a group, particularizing the previous notion of a Hopf brace with ${\sf C}={\sf Set}$ yields the notion of a skew brace, in which we have a pair of groups, $(G,\cdot)$ and $(G,\circ)$, satisfying the equation 
	\[a\circ(b\cdot c)=(a\circ b)\cdot a^{-1}\cdot(a\circ c), \textnormal{ for all }a,b,c\in G,\]
	which corresponds to \eqref{compatHbrace} in this setting. We will denote skew braces by $(G,\cdot,\circ)$ and, if we want to refer only to the first group structure involved, we will identify $(G,\cdot)=G_{\cdot}$, while for the second we will use $(G,\circ)=G_{\circ}$.  If $G_{\cdot}$  is an abelian group, then a skew brace is a brace in the sense of Rump \cite{Rump}.
\end{example}
\begin{example}\label{galg}
Let $\mathbb{K}$ be a field. By linearization, it is well know that if $(G,\cdot,\circ)$ is a skew brace, then the respective group algebras of $G_{\cdot}$ and $G_{\circ}$, $\mathbb{K}[G_{\cdot}]$ and $\mathbb{K}[G_{\circ}]$, constitute a Hopf brace in ${\sf C}={}_{\mathbb{K}}{\sf Vect}$, denoted by $\mathbb{K}[\mathbb{G}]=(\mathbb{K}[G_{\cdot}],\mathbb{K}[G_{\circ}])$.
\end{example}
\begin{example}
	Any Hopf algebra $H=(H,\eta_{H},\mu_{H},\varepsilon_{H},\delta_{H},\lambda_{H})$ leads to a trivial Hopf brace in ${\sf C}$: $\mathbb{H}_{{\rm triv}}=(H,H)$. That is because, when both Hopf algebras coincide in a Hopf brace structure, the compatibility condition \eqref{compatHbrace} always holds.
\end{example}

Given a Hopf brace $\mathbb{H}=(H_{1},H_{2})$, it is satisfied that $\eta_{H}^{1}=\eta_{H}^{2}$ \cite[Remark 1.3]{AGV}. Therefore, from now on we will denote both units by $\eta_{H}$. Besides, $(H_{1},\Gamma_{H_{1}})$ is a left $H_{2}$-module algebra \cite[Lemma 1.8]{AGV} and, reciprocally, this property characterizes the Hopf brace structure. Indeed, given a pair of Hopf algebras sharing the same underlying coalgebra, $H_{1}=(H,\eta_{H}^{1},\mu_{H}^{1},\varepsilon_{H},\delta_{H},\lambda_{H}^{1})$ and $H_{2}=(H,\eta_{H}^{2},\mu_{H}^{},\varepsilon_{H},\delta_{H},\lambda_{H}^{2})$, then 
\begin{center}
	$\mathbb{H}=(H_{1},H_{2})\in{\sf HBr}\iff \mu_{H}^{1}\circ(\Gamma_{H_{1}}\otimes\Gamma_{H_{1}})\circ(H\otimes c_{H,H}\otimes H)\circ(\delta_{H}\otimes H\otimes H)=\Gamma_{H_{1}}\circ(H\otimes\mu_{H}^{1}),$
\end{center}
as can be consulted in \cite[Theorem 1.16]{RGONRAMOS}. Moreover, by \eqref{def_GammaH1}, coassociativity of $\delta_{H}$, associativity of $\mu_{H}^{1}$ and \eqref{antipode} for $H_{1}$, we obtain the following expression for $\mu_{H}^{2}$:
\begin{gather}\label{mu2-exp}
	\mu_{H}^{2}=\mu_{H}^{1}\circ(H\otimes \Gamma_{H_{1}})\circ(\delta_{H}\otimes H).
\end{gather}

Furthermore, the cocommutativity class also provides a sufficient condition for $\Gamma_{H_{1}}$ to be a coalgebra morphism.
\begin{theorem}
	Let $\mathbb{H}=(H_{1},H_{2})\in{\sf HBr}$. If $(H_{1},\Gamma_{H_{1}})$ belongs to the cocommutativity class of $H_{2}$, then $\Gamma_{H_{1}}$ is a coalgebra morphism.
\end{theorem}
\begin{proof}
	It can be easily verified that $\varepsilon_{H}\circ\Gamma_{H_{1}}=\varepsilon_{H}\otimes\varepsilon_{H}$. Therefore, it only remains to show that $\Gamma_{H_{1}}$ preserves the coproduct. Indeed,
	\begin{itemize}
		\itemindent=-32pt 
		\item[ ]$\hspace{0.38cm}\delta_{H}\circ\Gamma_{H_{1}}$
		\item[] $=(\mu_{H}^{1}\otimes\mu_{H}^{1})\circ(H\otimes c_{H,H}\otimes H)\circ((\delta_{H}\circ\lambda_{H}^{1})\otimes(\delta_{H}\circ\mu_{H}^{2}))\circ(\delta_{H}\otimes H)$ {\footnotesize (by the condition of coalgebra}
		\item[]$\hspace{0.38cm}${\footnotesize morphism for $\mu_{H}^{1}$)}
		\item[] $=(\mu_{H}^{1}\otimes\mu_{H}^{1})\circ(H\otimes c_{H,H}\otimes H)\circ((c_{H,H}\circ(\lambda_{H}^{1}\otimes\lambda_{H}^{1})\circ\delta_{H})\otimes((\mu_{H}^{2}\otimes\mu_{H}^{2})\circ(H\otimes c_{H,H}\otimes H)$
		\item[]$\hspace{0.38cm}\circ(\delta_{H}\otimes\delta_{H})))\circ(\delta_{H}\otimes H)$ {\footnotesize(by the condition of coalgebra morphism for $\mu_{H}^{2}$ and \eqref{a-antip2})}
		\item[]$=(H\otimes (\mu_{H}^{1}\circ(\lambda_{H}^{1}\otimes H)))\circ(((\Gamma_{H_{1}}\otimes H)\circ(H\otimes c_{H,H})\circ((c_{H,H}\circ\delta_{H})\otimes H))\otimes\mu_{H}^{2})$
		\item[] $\hspace{0.38cm}\circ(H\otimes c_{H,H}\otimes H)\circ(\delta_{H}\otimes\delta_{H})$ {\footnotesize(by coassociativity of $\delta_{H}$, naturality of $c$ and \eqref{def_GammaH1})}
		\item[] $=(H\otimes (\mu_{H}^{1}\circ(\lambda_{H}^{1}\otimes H)))\circ(((\Gamma_{H_{1}}\otimes H)\circ(H\otimes c_{H,H})\circ(\delta_{H}\otimes H))\otimes\mu_{H}^{2})\circ(H\otimes c_{H,H}\otimes H)\circ(\delta_{H}\otimes\delta_{H})$ 
		\item[]$\hspace{0.38cm}${\footnotesize(by \eqref{ccclass} for $(H_{1},\Gamma_{H_{1}})$)}
		\item[] $=(\Gamma_{H_{1}}\otimes\Gamma_{H_{1}})\circ(H\otimes c_{H,H}\otimes H)\circ(\delta_{H}\otimes\delta_{H})$ {\footnotesize (by coassociativity of $\delta_{H}$ and naturality of $c$)}.\qedhere
	\end{itemize}
\end{proof}

By the naturality of $c$, the coassociativity of $\delta_{H}$ and the associativity of $\mu_{H}^{1}$, we obtain that
\begin{align*}
	&\mu_{H}^{1}\circ(\mu_{H}^{2}\otimes\Gamma_{H_{1}})\circ(H\otimes c_{H,H}\otimes H)\circ(\delta_{H}\otimes H\otimes H)\\=\,&\mu_{H}^{1}\circ(\Gamma'_{H_{1}}\otimes\mu_{H}^{2})\circ(H\otimes c_{H,H}\otimes H)\circ(\delta_{H}\otimes H\otimes H),
\end{align*}
where \begin{equation}\label{GammaprimadefH1}\Gamma'_{H_{1}}\coloneqq \mu_{H}^{1}\circ(\mu_{H}^{2}\otimes\lambda_{H}^{1})\circ(H\otimes c_{H,H})\circ(\delta_{H}\otimes H).\end{equation}

Therefore, \eqref{compatHbrace} is equivalent to
\begin{gather}\label{compatHprimabrace}
\mu_{H}^{2}\circ(H\otimes\mu_{H}^{1})=\mu_{H}^{1}\circ(\Gamma'_{H_{1}}\otimes\mu_{H}^{2})\circ(H\otimes c_{H,H}\otimes H)\circ(\delta_{H}\otimes H\otimes H).
\end{gather}

The properties satisfied by $\Gamma'_{H_{1}}$ are very similar to those satisfied by $\Gamma_{H_{1}}$. It is straightforward to show that $\mu_{H}^{2}$ admits the following expression in terms of $\Gamma'_{H_{1}}$:
\begin{gather}\label{mu2prime-exp}
	\mu_{H}^{2}=\mu_{H}^{1}\circ(\Gamma'_{H_{1}}\otimes H)\circ(H\otimes c_{H,H})\circ(\delta_{H}\otimes H).
\end{gather}

Besides, $(H_{1},\Gamma'_{H_{1}})$ is also a left $H_{2}$-module algebra \cite[Theorem 6.4]{BRZ}, and this property characterizes the Hopf brace structure as can be seen as follows. 
\begin{theorem}\label{equiv2HBr}
	Let $H_{1}$ and $H_{2}$ be two Hopf algebras in ${\sf C}$ with the same underlying coalgebra, $$H_{i}=(H,\eta_{H}^{i},\mu_{H}^{i},\varepsilon_{H},\delta_{H},\lambda_{H}^{i})$$ for $i=1,2$. Consider the morphism $\Gamma'_{H_{1}}$ defined as in \eqref{GammaprimadefH1}. The following statements are equivalent:
	\begin{itemize}
		\item[(i)] $\mathbb{H}=(H_{1},H_{2})$ is a Hopf brace in ${\sf C}$,
		\item[(ii)] $\Gamma'_{H_{1}}\circ(H\otimes\mu_{H}^{1})=\mu_{H}^{1}\circ (\Gamma'_{H_{1}}\otimes\Gamma'_{H_{1}})\circ(H\otimes c_{H,H}\otimes H)\circ(\delta_{H}\otimes H\otimes H)$.
	\end{itemize}
\end{theorem}
\begin{proof}
	It only remains to prove (ii)$\implies$(i). Let us show that \eqref{compatHprimabrace} holds. Indeed,
	\begin{itemize}
		\itemindent=-32pt 
		\item[ ]$\hspace{0.38cm} \mu_{H}^{1}\circ(\Gamma'_{H_{1}}\otimes\mu_{H}^{2})\circ(H\otimes c_{H,H}\otimes H)\circ(\delta_{H}\otimes H\otimes H)$
		\item[]$=\mu_{H}^{1}\circ(\Gamma'_{H_{1}}\otimes(\mu_{H}^{1}\circ(\Gamma'_{H_{1}}\otimes H)\circ(H\otimes c_{H,H})\circ(\delta_{H}\otimes H)))\circ(H\otimes c_{H,H}\otimes H)\circ(\delta_{H}\otimes H\otimes H)$
		\item[]$\hspace{0.38cm}${\footnotesize (by \eqref{mu2prime-exp})}
		\item[] $=\mu_{H}^{1}\circ((\mu_{H}^{1}\circ (\Gamma'_{H_{1}}\otimes\Gamma'_{H_{1}})\circ(H\otimes c_{H,H}\otimes H)\circ(\delta_{H}\otimes H\otimes H))\otimes H)\circ(H\otimes((H\otimes c_{H,H})$
		\item[]$\hspace{0.38cm}\circ(c_{H,H}\otimes H)))\circ(\delta_{H}\otimes H\otimes H)$ {\footnotesize(by associativity of $\mu_{H}^{1}$, the naturality of $c$ and coassociativity of $\delta_{H}$)}
		\item[] $=\mu_{H}^{1}\circ(\Gamma'_{H_{1}}\otimes H)\circ(H\otimes((\mu_{H}^{1}\otimes H)\circ(H\otimes c_{H,H})\circ(c_{H,H}\otimes H)))\circ(\delta_{H}\otimes H\otimes H)$ {\footnotesize (by (ii))}
		\item[] $=\mu_{H}^{1}\circ(\Gamma'_{H_{1}}\otimes H)\circ(H\otimes c_{H,H})\circ(\delta_{H}\otimes\mu_{H}^{1})$ {\footnotesize (by naturality of $c$)}
		\item[] $=\mu_{H}^{2}\circ(H\otimes\mu_{H}^{1})$ {\footnotesize (by \eqref{mu2prime-exp})}.\qedhere
	\end{itemize}
\end{proof}
To conclude, the cocommutativity class gives a sufficient condition for $\Gamma'_{H_{1}}$ to be a coalgebra morphism.
\begin{theorem}
	Let $\mathbb{H}=(H_{1},H_{2})\in{\sf HBr}$. If $(H_{1},\Gamma'_{H_{1}})$ belongs to the cocommutativity class of $H_{2}$, then $\Gamma'_{H_{1}}$ is a coalgebra morphism.
\end{theorem}
\begin{proof}
	Let $H=(H,\eta_{H},\mu_{H},\varepsilon_{H},\delta_{H},\lambda_{H})$ be a Hopf algebra in ${\sf C}$. Note that, in the case that $\mathbb{H}=\mathbb{H}_{{\rm triv}}=(H,H)$, the action $\Gamma'_{H}=\varphi_{H}^{{\rm ad}}$. Therefore, we omit a detailed proof here for this result since the steps to follow are analogous to those in Example \ref{ccclass_adjoint}. However, unlike Example \ref{ccclass_adjoint}, the condition is sufficient but not necessary.
\end{proof}

\section{From modules over a Hopf brace to modules over an algebra}\label{sec_equivalence}
The weakest notion for the category of modules over a Hopf brace that can be found in the literature is the one introduced by González Rodríguez in \cite{RGON} in order to prove the Fundamental Theorem of Hopf modules in the Hopf brace setting.
\begin{definition}\label{RGONdef}
Let $\mathbb{H}=(H_{1},H_{2})$ be a Hopf brace in {\sf C}. A left module over $\mathbb{H}$ is a triple $(M,\varphi_{M}^{1},\varphi_{M}^{2})$ where $(M,\varphi_{M}^{1})$ is a left $H_{1}$-module, $(M,\varphi_{M}^{2})$ is a left $H_{2}$-module and the compatibility condition 
\begin{gather}\label{compatmodH}
\varphi_{M}^{2}\circ(H\otimes\varphi_{M}^{1})=\varphi_{M}^{1}\circ(\mu_{H}^{2}\otimes\Gamma_{M})\circ(H\otimes c_{H,H}\otimes M)\circ(\delta_{H}\otimes H\otimes M)
\end{gather}
holds, where $\Gamma_{M}$ is defined by 
\begin{gather}\label{GammaMdef}
\Gamma_{M}\coloneqq\varphi_{M}^{1}\circ(\lambda_{H}^{1}\otimes\varphi_{M}^{2})\circ(\delta_{H}\otimes M).
\end{gather}
	
	Given left $\mathbb{H}$-modules  $(M,\varphi_{M}^1, \varphi_{M}^{2})$  and  $(N,\varphi_{N}^1, \varphi_{N}^{2})$, a morphism $$f\colon (M,\varphi_{M}^1, \varphi_{M}^{2})\rightarrow (N,\varphi_{N}^1, \varphi_{N}^{2})$$  is said to be a morphism of left $\mathbb{H}$-modules if  $f\colon (M,\varphi_{M}^{1})\rightarrow(N,\varphi_{N}^{1})$ is $H_{1}$-linear and $f\colon (M,\varphi_{M}^{2})\rightarrow(N,\varphi_{N}^{2})$ is $H_{2}$-linear. Left ${\mathbb H}$-modules constitute a category which we will denote by ${}_{\mathbb{H}}{\sf Mod}$. 
\end{definition}
\begin{remark}
	Note that, by the naturality of $c$, the coassociativity of $\delta_{H}$ and \eqref{actionprod} for $(M,\varphi_{M}^{1})$, we obtain that
	\begin{align*}
		&\varphi_{M}^{1}\circ(\mu_{H}^{2}\otimes\Gamma_{M})\circ(H\otimes c_{H,H}\otimes M)\circ(\delta_{H}\otimes H\otimes M)\\=\,&\varphi_{M}^{1}\circ(\Gamma'_{H_{1}}\otimes\varphi_{M}^{2})\circ(H\otimes c_{H,H}\otimes M)\circ(\delta_{H}\otimes H\otimes M),
	\end{align*}
	where $\Gamma'_{H_{1}}$ is defined as in \eqref{GammaprimadefH1}. As a result, \eqref{compatmodH} is equivalent to
	\begin{gather}\label{compatmodHprime}
		\varphi_{M}^{2}\circ(H\otimes\varphi_{M}^{1})=\varphi_{M}^{1}\circ(\Gamma'_{H_{1}}\otimes\varphi_{M}^{2})\circ(H\otimes c_{H,H}\otimes M)\circ(\delta_{H}\otimes H\otimes M).
	\end{gather}
\end{remark}
\begin{example}\label{ex_HmodHbrace}
	Given $\mathbb{H}=(H_{1},H_{2})$ a Hopf brace in ${\sf C}$, the triple $(H,\varphi_{H}^{1}=\mu_{H}^{1},\varphi_{H}^{2}=\mu_{H}^{2})$ is an example of left $\mathbb{H}$-module. For this situation, $\Gamma_{H}$ defined in \eqref{GammaMdef} coincides with the action $\Gamma_{H_{1}}$ defined in \eqref{def_GammaH1}, and then \eqref{compatmodH} follows by \eqref{compatHbrace}.
\end{example}

\begin{example}
	Let $H=(H,\eta_{H},\mu_{H},\varepsilon_{H},\delta_{H},\lambda_{H})$ be a Hopf algebra in ${\sf C}$ and consider $\mathbb{H}_{\textnormal{triv}}=(H,H)$ the trivial Hopf brace constructed from $H$. If $(M,\varphi_{M})$ is a left $H$-module, then $(M,\varphi_{M},\varphi_{M})$ is a left $\mathbb{H}_{\textnormal{triv}}$-module. In this situation, $\Gamma_{M}=\varepsilon_{H}\otimes M$.
\end{example}

Modules over Hopf braces rely on two structures of left $H_2$-modules. On the one hand, given a left $\mathbb H$-module $(M, \varphi_M^1, \varphi_M^2)$ we have that $(M, \varphi_M^2)$ is a left $H_2$-module. On the other hand, as can be consulted in \cite[Lemma 2.11]{VGRRProj}, $(M, \Gamma_M)$ is a left $H_2$-module, where $\Gamma_M$ is defined as in \eqref{GammaMdef}. Moreover, following the proof of \eqref{mu2-exp}, we obtain an expression for $\varphi_M^2$ in terms of $\Gamma_M$ and $\varphi_M^1$:
\begin{gather}\label{varphi2expression}
	\varphi_{M}^{2}=\varphi_{M}^{1}\circ(H\otimes\Gamma_{M})\circ(\delta_{H}\otimes M).
\end{gather}

These $H_2$-actions together with equality \eqref{varphi2expression} encode the conditions for $(M, \varphi_M^1, \varphi_M^2)$ to be a left $\mathbb H$-module. To obtain such conditions,  let us define the category of left ${\mathbb H}$-actions, denoted by  $_{\mathbb H}\sf{Ac}$ and introduced in \cite{Ramos}. The objects in $_{\mathbb H}\sf{Ac}$ are triples  $(M, \varphi_M, \phi_M)$ such that  $(M, \phi_M)$ is a left $H_2$-module and $(M, \varphi_M)$ is a left $H_1$-module whose action $$\varphi_M:(H\otimes M, \phi_{H\otimes M})\to (M, \phi_M)$$ is a morphism of left $H_2$-modules, where $\phi_{H\otimes M} \coloneqq (\Gamma_{H_1}\otimes \phi_M)\circ (H\otimes c_{H,H}\otimes M)\circ (\delta_H\otimes M)$, i.e., 
\begin{equation}
\label{linearphi}
\phi_{M}\circ (H\otimes \varphi_{M})=\varphi_{M}\circ (\Gamma_{H_{1}}\otimes \phi_{M})\circ (H\otimes c_{H,H}\otimes M)\circ (\delta_{H}\otimes H\otimes M)
\end{equation}

Morphisms in  $_{\mathbb H}\sf{Ac}$ are morphisms of left $H_1$-modules and of left $H_2$-modules. This category is isomorphic to $_{\mathbb H}\sf{Mod}$. Indeed, define the functor 
$${\sf U}:{}_{\mathbb H}\sf{Ac}\longrightarrow {}_{\mathbb H}\sf{Mod}$$
 in the following way:
$${\sf U}(M, \varphi_M, \phi_M) = (M, \varphi_{M}^1 = \varphi_M, \varphi_M^2 = \varphi_M\circ (H\otimes \phi_M)\circ (\delta_H\otimes M))$$ on objects and as the identity on morphisms.

 First of all we will see that $(M, \varphi_M^2)$ is a left $H_2$-module. Indeed:
\begin{itemize}
	\itemindent=-32pt 
	\item[ ]$\hspace{0.38cm} \varphi_M^2\circ (H\otimes \varphi_M^2)$ 
	\item[] $= \varphi_M\circ (H\otimes \phi_M)\circ (\delta_H\otimes (\varphi_M\circ (H\otimes \phi_M)\circ (\delta_H\otimes M)))$ {\footnotesize (by definition of $\varphi_M^2$)}
	\item[] $=\varphi_{M}\circ (H\otimes(\varphi_{M}\circ (\Gamma_{H_{1}}\otimes\phi_{M})\circ(H\otimes c_{H,H}\otimes M)\circ(\delta_{H}\otimes H\otimes M)))\circ (\delta_{H}\otimes ((H\otimes\phi_{M})\circ(\delta_{H}\otimes M)))$
	\item[]$\hspace{0.38cm}${\footnotesize (by \eqref{linearphi})}
	\item[]$=\varphi_{M}\circ((\mu_{H}^{1}\circ(H\otimes\Gamma_{H_{1}})\circ(\delta_{H}\otimes H))\otimes(\phi_{M}\circ(\mu_{H}^{2}\otimes M)))\circ(((H\otimes c_{H,H}\otimes H)\circ(\delta_{H}\otimes\delta_{H}))\otimes M)$ 
	\item[]$\hspace{0.38cm}${\footnotesize (by coassociativity of $\delta_{H}$, the condition of left $H_{1}$-module for $(M,\varphi_{M})$ and the condition of left $H_{2}$-module}
	\item[]$\hspace{0.38cm}${\footnotesize for $(M,\phi_{M})$)}
	\item[]$=\varphi_{M}\circ(H\otimes\phi_{M})\circ (((\mu_{H}^{2}\otimes\mu_{H}^{2})\circ(H\otimes c_{H,H}\otimes H)\circ (\delta_{H}\otimes\delta_{H}))\otimes M)$ {\footnotesize (by \eqref{mu2-exp})}
	\item[] $ = \varphi_M\circ (H\otimes \phi_M)\circ (\delta_H\otimes M)\circ (\mu_{H}^{2}\otimes M)$ {\footnotesize (by the coalgebra morphism condition for $\mu_{H}^{2}$)}
	\item[] $ = \varphi_M^2\circ (\mu_{H}^{2}\otimes M)$ {\footnotesize (by the definition of $\varphi_M^2$)}.
\end{itemize}

Since $\varphi_M^2\circ (\eta_H\otimes M) ={\rm id}_{M}$ by straightforward calculations, we have that $(M, \varphi_M^2)$ is a left $H_2$-module. 

Now, defining $\Gamma_M$ as in \eqref{GammaMdef}, we obtain that:
\begin{itemize}
	\itemindent=-32pt
	\item[] $\hspace{0.38cm}\Gamma_M$
	\item[] $ = \varphi_M^1 \circ (\lambda_{H}^1\otimes \varphi_M^2)\circ (\delta_{H}\otimes M)$
	\item[] $ = \varphi_M\circ (\lambda_{H}^1\otimes (\varphi_M\circ (H\otimes \phi_{M})\circ (\delta_{H}\otimes M)))
	\circ (\delta_H\otimes M)$ {\footnotesize (by definition of $\varphi_M^2$)}
	\item[] $ = \phi_M$ {\footnotesize (by coassociativity of $\delta_{H}$,  the condition of left $H_1$-module for $(M, \varphi_M)$ and (\ref{antipode})).}
\end{itemize}

Thus $\Gamma_M = \phi_M$, and from this equality and the $H_2$-linearity of $\varphi_M$ follows condition \eqref{compatmodH}. Indeed:
\begin{itemize}
	\itemindent=-32pt
	\item[] $\hspace{0.38cm} \varphi_M^2\circ (H\ot \varphi_M^1)$
	\item[] $ = \varphi_M\circ (H\ot \phi_M)\circ (\delta_H\ot \varphi_M)$ {\footnotesize (by the definition of $\varphi_M$ and equality $\Gamma_M = \phi_M$)}
	\item[] $ = \varphi_M\circ (H\ot (\varphi_M\circ (\Gamma_{H_1}\ot \Gamma_M)\circ (H\ot c_{H,H}\ot M)\circ (\delta_H\ot H\ot M)))$ {\footnotesize (by (\ref{linearphi}))}
	\item[] $ = \varphi_M^1\circ (\mu_{H}^2\ot \Gamma_M)\circ (H\ot c_{H,H}\ot M)\circ (\delta_H\ot H\ot M)$ {\footnotesize (by coassociativity of $\delta_{H}$, \eqref{mu2-exp} and the condition}
	\item[] $\hspace{0.38cm}$ {\footnotesize   of $H_1$-module for $(M, \varphi_M)$)}.
\end{itemize}

As a consequence, $(M, \varphi_M^1, \varphi_M^2)$ is a left ${\mathbb H}$-module. Finally observe that, if $$f:(M, \varphi_M, \phi_M)\to (N, \varphi_N, \phi_N)$$ is a morphism of left $H_1$-modules and  left $H_2$-modules, then $f:(M, \varphi_M^2)\to (N, \varphi_N^2)$ is also a left $H_2$-module morphism. 

Now consider $${\sf V}:{}_{\mathbb H}\sf{Mod}\longrightarrow{} _{\mathbb H}\sf{Ac}$$ the functor defined on objects as ${\sf V}(M, \varphi_M^1, \varphi_M^2) = (M, \varphi_M=\varphi_M^1, \phi_M = \Gamma_M)$. As we have already mentioned, we know that $(M,\Gamma_{M})$ is a left $H_{2}$-module. Moreover, also from \cite[Lemma 2.11]{VGRRProj}, the equality
\begin{equation}
	\label{GMH1}
	\Gamma_{M}\circ (H\otimes \varphi_{M}^1)=\varphi_{M}^1\circ (\Gamma_{H_1}\otimes \Gamma_{M})\circ (H\otimes c_{H,H}\otimes M)\circ (\delta_{H}\otimes H\otimes M)
\end{equation}
holds. Therefore, $(M, \varphi_M^1, \Gamma_M)$ is an object in  ${}_{\mathbb H}\sf{Ac}$. On morphisms ${\sf V}$ is defined by the identity, because if $f: (M, \varphi_M^1, \varphi_M^2)\to (N, \varphi_N^1, \varphi_M^2)$ is a morphism in ${}_{\mathbb H}{\sf Mod}$, then it is clear that $f:(M, \Gamma_M)\to (N, \Gamma_N)$ is a morphism of left $H_2$-modules. Finally, by straightforward calculations we have that these two functors are inverse. As a consequence, given a left $H_1$-module and left $H_2$-module $(M, \varphi_M^1, \varphi_M^2)$, and defining $\Gamma_M$ as in \eqref{GammaMdef}, we have that equality \eqref{GMH1} characterizes left $\mathbb H$-modules. On the other hand, given a triple $(M, \varphi_M, \Gamma_M)$  satisfying equality \eqref{GMH1}, equality \eqref{varphi2expression} characterizes the module structure over $H_2$ for a left $\mathbb H$-module (and it is, so to say, a reciprocal equality to \eqref{GammaMdef}). 

\begin{example}\label{ex-mod-Rump}
	Let ${\mathbb H} = (H_1, H_2)$ be a Hopf brace and let $(M,\phi_{M})$ be a left $H_{2}$-module. The triple $(M,\varphi_{M}^{1}=\varepsilon_{H}\otimes M, \phi_{M})$ is an object in ${}_{\mathbb H}\sf{Ac}$ and thus it can be regarded as a left $\mathbb{H}$-module. As a consequence, $_{H_2}\sf{Mod}$ is a subcategory of $_{\mathbb H}\sf{Mod}$. 
	
	Now let $(G, \cdot, \circ)$ be a  brace in the sense of Rump  (this is, a skew brace whose first group structure is abelian). In \cite{Rump} the category of modules over a brace is defined, and it is proved to be equivalent to the category of left modules over $(G, \circ)$. Let $\mathbb K$ be an arbitrary field, and consider the corresponding Hopf brace ${\mathbb K[{\mathbb G}]} = ({\mathbb K}[G_{\cdot}], {\mathbb K}[G_{\circ}])$ defined in Example \ref{galg}. In this case modules in the sense of Rump over ${\mathbb K[{\mathbb G}]}$ are $_{{\mathbb K}[G_{\circ}]}\sf{Mod}$ that, as we have mentioned, it is a subcategory of $_{\mathbb K[{\mathbb G}]}{\sf Mod}$.
\end{example}

Moreover, as  it is discussed in \cite[Theorem 2.12]{VGRRProj}, ${}_{\mathbb{H}}{\sf Mod}$ is a symmetric monoidal category when $\mathbb{H}$ is a cocommutative Hopf brace. Nevertheless, the assumption of cocommutativity on $\mathbb{H}$ can be relaxed by once more appealing to the cocommutativity class, while still obtaining that the category is monoidal.
\begin{definition}
	Let $\mathbb{H}$ be a Hopf brace in ${\sf C}$. We will denote by ${}_{\mathbb{H}}{\sf Mod}^{{\sf cc}}$ the full subcategory of ${}_{\mathbb{H}}{\sf Mod}$ whose objects satisfy that $(M,\Gamma_{M})$ belongs to the cocommutativity class of $H_{2}$, that is, for the objects in this subcategory we have that
	\begin{gather}\label{ccclassGammaM}
		(\Gamma_{M}\otimes H)\circ(H\otimes c_{H,M})\circ(\delta_{H}\otimes M)=(\Gamma_{M}\otimes H)\circ(H\otimes c_{H,M})\circ((c_{H,H}\circ\delta_{H})\otimes M).
	\end{gather}
	
	Then, the superscript ${\sf cc}$ in the notation of this subcategory is chosen in reference to the term ``cocommutativity class''. 
	
	If $\mathbb{H}$ is assumed to be cocommutative, then \eqref{ccclassGammaM} holds and, consequently,  ${}_{\mathbb{H}}{\sf Mod}^{{\sf cc}}={}_{\mathbb{H}}{\sf Mod}$.
\end{definition}
\begin{theorem}
	Let $\mathbb{H}=(H_{1},H_{2})$ be a Hopf brace in ${\sf C}$. The category ${}_{\mathbb{H}}{\sf Mod}^{{\sf cc}}$ is monoidal with unit object $(K,\varepsilon_{H}\otimes K,\varepsilon_{H}\otimes K)$ and tensor functor defined as follows:
	\begin{align*}
		\otimes\colon {}_{\mathbb{H}}{\sf Mod}^{{\sf cc}}\times {}_{\mathbb{H}}{\sf Mod}^{{\sf cc}}&\longrightarrow {}_{\mathbb{H}}{\sf Mod}^{{\sf cc}}\\((M,\varphi_{M}^{1},\varphi_{M}^{2}),(N,\varphi_{N}^{1},\varphi_{N}^{2}))&\longmapsto(M\otimes N,\varphi_{M\otimes N}^{1},\varphi_{M\otimes N}^{2}),
	\end{align*}
	where
	\begin{gather}\label{usualtensoraction}
		\varphi_{M\otimes N}^{i}=(\varphi_{M}^{i}\otimes\varphi_{N}^{i})\circ(H\otimes c_{H,M}\otimes N)\circ(\delta_{H}\otimes M\otimes N),\textnormal{ for all }i=1,2.
	\end{gather}
\end{theorem}
\begin{proof}
	It is a classical result in the theory of Hopf algebras that the category of modules over a Hopf algebra $H$ is monoidal, where, if $(M,\varphi_{M})$ and $(N,\varphi_{N})$ are modules over $H$, then the tensor functor is given by $(M\otimes N,\varphi_{M\otimes N})$ with
	\begin{gather*}\label{tensorusualactionhopf}
		\varphi_{M\otimes N}\coloneqq(\varphi_{M}\otimes\varphi_{N})\circ(H\otimes c_{H,M}\otimes N)\circ(\delta_{H}\otimes M\otimes N).\end{gather*}
	Therefore, $(M\otimes N, \varphi_{M\otimes N}^1)$ is a left $H_{1}$-module and $(M\otimes N, \varphi_{M\otimes N}^2)$ is a left $H_{2}$-module.
	
	Let us show that $(M\otimes N,\varphi_{M\otimes N}^{1},\varphi_{M\otimes N}^{2})$ satisfies \eqref{compatmodH}. To do so, note at first that
	\begin{gather}\label{Gammamodtensor}
		\Gamma_{M\otimes N}=(\Gamma_{M}\otimes \Gamma_{N})\circ(H\otimes c_{H,M}\otimes N)\circ(\delta_{H}\otimes M\otimes N).
	\end{gather}
	Indeed,
	\begin{itemize}
		\itemindent=-32pt 
		\item[ ]$\hspace{0.38cm}\Gamma_{M\otimes N}$
		\item[]$=(\varphi_{M}^{1}\otimes\varphi_{N}^{1})\circ(H\otimes c_{H,M}\otimes N)\circ((\delta_{H}\circ\lambda_{H}^{1})\otimes((\varphi_{M}^{2}\otimes\varphi_{N}^{2})\circ(H\otimes c_{H,M}\otimes N)$
		\item[]$\hspace{0.38cm}\circ (\delta_{H}\otimes M\otimes N)))\circ (\delta_{H}\otimes M\otimes N)$ {\footnotesize (by \eqref{GammaMdef} for $(M\otimes N,\varphi_{M\otimes N}^{1},\varphi_{M\otimes N}^{2})$)}
		\item[]$=(\varphi_{M}^{1}\otimes\varphi_{N}^{1})\circ(H\otimes c_{H,M}\otimes N)\circ(((\lambda_{H}^{1}\otimes\lambda_{H}^{1})\circ c_{H,H}\circ\delta_{H})\otimes((\varphi_{M}^{2}\otimes\varphi_{N}^{2})\circ(H\otimes c_{H,M}\otimes N)$
		\item[]$\hspace{0.38cm}\circ (\delta_{H}\otimes M\otimes N)))\circ (\delta_{H}\otimes M\otimes N)$ {\footnotesize (by \eqref{a-antip2} for $H_{1}$)}
		\item[]$=((\varphi_{M}^{1}\circ(\lambda_{H}^{1}\otimes\varphi_{M}^{2}))\otimes(\varphi_{N}^{1}\circ(\lambda_{H}^{1}\otimes\varphi_{N}^{2})))\circ(((H\otimes H\otimes c_{H,M})\circ (H\otimes c_{H,H}\otimes M)$
		\item[]$\hspace{0.38cm}\circ(c_{H,H}\otimes H\otimes M)\circ(H\otimes\delta_{H}\otimes M))\otimes H\otimes N)\circ (((\delta_{H}\otimes c_{H,M})\circ (\delta_{H}\otimes M))\otimes N)$ {\footnotesize (by coassociativity}
		\item[]$\hspace{0.38cm}${\footnotesize of $\delta_{H}$ and the naturality of $c$)}
		\item[]$=(M\otimes (\varphi_{N}^{1}\circ(\lambda_{H}^{1}\otimes \varphi_{N}^{2})))\circ (((\Gamma_{M}\otimes H)\circ(H\otimes c_{H,M})\circ ((c_{H,H}\circ\delta_{H})\otimes M))\otimes H\otimes N)$
		\item[]$\hspace{0.38cm}\circ(H\otimes c_{H,M}\otimes N)\circ(\delta_{H}\otimes M\otimes N)$ {\footnotesize (by naturality of $c$ and \eqref{GammaMdef} for $(M,\varphi_{M}^{1},\varphi_{M}^{2})$)}
		\item[]$=(M\otimes (\varphi_{N}^{1}\circ(\lambda_{H}^{1}\otimes \varphi_{N}^{2})))\circ (((\Gamma_{M}\otimes H)\circ(H\otimes c_{H,M})\circ (\delta_{H}\otimes M))\otimes H\otimes N)$
		\item[]$\hspace{0.38cm}\circ(H\otimes c_{H,M}\otimes N)\circ(\delta_{H}\otimes M\otimes N)$ {\footnotesize (by \eqref{ccclassGammaM})}
		\item[] $=(\Gamma_{M}\otimes\Gamma_{N})\circ(H\otimes c_{H,M}\otimes N)\circ(\delta_{H}\otimes M\otimes N)$ {\footnotesize (by coassociativity of $\delta_{H}$, the naturality of $c$ and}
		\item[]$\hspace{0.38cm}${\footnotesize \eqref{GammaMdef} for $(N,\varphi_{N}^{1},\varphi_{N}^{2})$).}
	\end{itemize}
	
	Then, \eqref{compatmodH} for $(M\otimes N,\varphi_{M\otimes N}^{1},\varphi_{M\otimes N}^{2})$ follows by
	\begin{itemize}
		\itemindent=-32pt 
		\item[ ]$\hspace{0.38cm}\varphi_{M\otimes N}^{1}\circ(\mu_{H}^{2}\otimes \Gamma_{M\otimes N})\circ(H\otimes c_{H,H}\otimes M\otimes N)\circ(\delta_{H}\otimes H\otimes M\otimes N)$
		\item[]$=(\varphi_{M}^{1}\otimes\varphi_{N}^{1})\circ(H\otimes c_{H,M}\otimes N)\circ((\delta_{H}\circ\mu_{H}^{2})\otimes ((\Gamma_{M}\otimes\Gamma_{N})\circ(H\otimes c_{H,M}\otimes N)$
		\item[]$\hspace{0.38cm}\circ(\delta_{H}\otimes M\otimes N)))\circ(H\otimes c_{H,H}\otimes M\otimes N)\circ(\delta_{H}\otimes H\otimes M\otimes N)$ {\footnotesize (by \eqref{usualtensoraction} and \eqref{Gammamodtensor})}
		\item[]$=((\varphi_{M}^{1}\circ (\mu_{H}^{2}\otimes \Gamma_{M}))\otimes(\varphi_{N}^{1}\circ(\mu_{H}^{2}\otimes\Gamma_{N})))\circ (H\otimes((H\otimes H\otimes c_{H,M}\otimes H\otimes H)\circ(((H\otimes c_{H,H})$
		\item[]$\hspace{0.38cm}\circ(c_{H,H}\otimes H)\circ(H\otimes c_{H,H})\circ (\delta_{H}\otimes H))\otimes((c_{H,M}\otimes H)\circ (H\otimes c_{H,M})\circ (c_{H,H}\otimes M)))$
		\item[]$\hspace{0.38cm}\circ (H\otimes H\otimes c_{H,H}\otimes H\otimes M))\otimes N)\circ(((\delta_{H}\otimes H)\circ\delta_{H})\otimes\delta_{H}\otimes M\otimes N)$ {\footnotesize (by naturality of $c$, the}
		\item[]$\hspace{0.38cm}${\footnotesize symmetric character of ${\sf C}$ and coassociativity of $\delta_{H}$)}
		\item[]$=((\varphi_{M}^{1}\circ (\mu_{H}^{2}\otimes M))\otimes(\varphi_{N}^{1}\circ(\mu_{H}^{2}\otimes\Gamma_{N})))\circ(H\otimes((H\otimes ((\Gamma_{M}\otimes H)\circ(H\otimes c_{H,M})$
		\item[]$\hspace{0.38cm}\circ((c_{H,H}\circ\delta_{H})\otimes M))\otimes H\otimes H)\circ (c_{H,H}\otimes ((c_{H,M}\otimes H)\circ (H\otimes c_{H,M})\circ (c_{H,H}\otimes M)))$
		\item[]$\hspace{0.38cm}\circ(H\otimes c_{H,H}\otimes H\otimes M))\otimes N)\circ(((\delta_{H}\otimes H)\circ\delta_{H})\otimes\delta_{H}\otimes M\otimes N)$ {\footnotesize (by naturality of $c$)}
		\item[]$=((\varphi_{M}^{1}\circ (\mu_{H}^{2}\otimes M))\otimes(\varphi_{N}^{1}\circ(\mu_{H}^{2}\otimes\Gamma_{N})))\circ(H\otimes((H\otimes ((\Gamma_{M}\otimes H)\circ(H\otimes c_{H,M})$
		\item[]$\hspace{0.38cm}\circ(\delta_{H}\otimes M))\otimes H\otimes H)\circ (c_{H,H}\otimes ((c_{H,M}\otimes H)\circ (H\otimes c_{H,M})\circ (c_{H,H}\otimes M)))$
		\item[]$\hspace{0.38cm}\circ(H\otimes c_{H,H}\otimes H\otimes M))\otimes N)\circ(((\delta_{H}\otimes H)\circ\delta_{H})\otimes\delta_{H}\otimes M\otimes N)$ {\footnotesize (by \eqref{ccclassGammaM})}
		\item[]$=((\varphi_{M}^{1}\circ(\mu_{H}^{2}\otimes\Gamma_{M})\circ(H\otimes c_{H,H}\otimes M)\circ(\delta_{H}\otimes H\otimes M))\otimes(\varphi_{N}^{1}\circ (\mu_{H}^{2}\otimes\Gamma_{N})$
		\item[]$\hspace{0.38cm}\circ (H\otimes c_{H,H}\otimes N)))\circ (H\otimes((H\otimes c_{H,M}\otimes H)\circ (c_{H,H}\otimes c_{H,M})\circ (H\otimes c_{H,H}\otimes M)$
		\item[]$\hspace{0.38cm}\circ(\delta_{H}\otimes H\otimes M))\otimes H\otimes N)\circ (\delta_{H}\otimes((H\otimes c_{H,M})\circ(\delta_{H}\otimes M))\otimes N)$ {\footnotesize (by naturality of $c$ and}
		\item[]$\hspace{0.38cm}${\footnotesize coassociativity of $\delta_{H}$)}
		\item[]$=((\varphi_{M}^{2}\circ(H\otimes\varphi_{M}^{1}))\otimes(\varphi_{N}^{1}\circ(\mu_{H}^{2}\otimes \Gamma_{N})\circ(H\otimes c_{H,H}\otimes N)\circ(\delta_{H}\otimes H\otimes N)))$
		\item[]$\hspace{0.38cm}\circ (H\otimes((H\otimes c_{H,M}\otimes H)\circ(c_{H,H}\otimes c_{H,M})) \otimes N)\circ(\delta_{H}\otimes\delta_{H}\otimes M\otimes N)$ {\footnotesize (by naturality of $c$ and}
		\item[]$\hspace{0.38cm}${\footnotesize \eqref{compatmodH} for $(M,\varphi_{M}^{1},\varphi_{M}^{2})$)}
		\item[]$=(\varphi_{M}^{2}\otimes (\varphi_{N}^{2}\circ(H\otimes \varphi_{N}^{1})))\circ (H\otimes((\varphi_{M}^{1}\otimes H)\circ (H\otimes c_{H,M})\circ (c_{H,H}\otimes M))\otimes H\otimes N)$
		\item[]$\hspace{0.38cm}\circ(\delta_{H}\otimes((H\otimes c_{H,M})\circ(\delta_{H}\otimes M))\otimes N)$ {\footnotesize (by \eqref{compatmodH} for $(N,\varphi_{N}^{1},\varphi_{N}^{2})$)} 
		\item[]$=\varphi_{M\otimes N}^{2}\circ(H\otimes\varphi_{M\otimes N}^{1})$ {\footnotesize (by naturality of $c$ and \eqref{usualtensoraction})}.
	\end{itemize}
	
	To conclude, it only remains to show that \eqref{ccclassGammaM} holds for $(M\otimes N,\varphi_{M\otimes N}^{1},\varphi_{M\otimes N}^{2})$.
	Indeed,
	\begin{itemize}
		\itemindent=-32pt 
		\item[ ]$\hspace{0.38cm}\Gamma_{M\otimes N}\circ(H\otimes c_{H,M\otimes N})\circ((c_{H,H}\circ\delta_{H})\otimes M\otimes N)$
		\item[]$=(((\Gamma_{M}\otimes\Gamma_{N})\circ (H\otimes c_{H,M}\otimes N)\circ (\delta_{H}\otimes M\otimes N))\otimes H)\circ (H\otimes((M\otimes c_{H,N})\circ (c_{H,M}\otimes N)))$
		\item[]$\hspace{0.38cm}\circ((c_{H,H}\circ\delta_{H})\otimes M\otimes N)$ {\footnotesize (by \eqref{Gammamodtensor})}
		\item[]$=(M\otimes ((\Gamma_{N}\otimes H)\circ(H\otimes c_{H,N})\circ (c_{H,H}\otimes N)))\circ(((\Gamma_{M}\otimes H)\circ(H\otimes c_{H,M})\circ((c_{H,H}\circ\delta_{H})$
		\item[]$\hspace{0.38cm}\otimes M))\otimes H\otimes N)\circ (((H\otimes c_{H,M})\circ(\delta_{H}\otimes M))\otimes N)$ {\footnotesize (by naturality of $c$ and coassociativity of $\delta_{H}$)}
		\item[]$=(M\otimes ((\Gamma_{N}\otimes H)\circ(H\otimes c_{H,N})\circ (c_{H,H}\otimes N)))\circ(((\Gamma_{M}\otimes H)\circ(H\otimes c_{H,M})\circ(\delta_{H}\otimes M))$
		\item[]$\hspace{0.38cm}\otimes H\otimes N)\circ (((H\otimes c_{H,M})\circ(\delta_{H}\otimes M))\otimes N)$ {\footnotesize (by \eqref{ccclassGammaM} for $(M,\Gamma_{M})$)}
		\item[]$=(\Gamma_{M}\otimes((\Gamma_{N}\otimes H)\circ (H\otimes c_{H,N})\circ ((c_{H,H}\circ\delta_{H})\otimes N)))\circ(H\otimes c_{H,M}\otimes N)\circ(\delta_{H}\otimes M\otimes N)$ 
		\item[]$\hspace{0.38cm}${\footnotesize (by coassociativity of $\delta_{H}$ and naturality of $c$)}
		\item[]$=(\Gamma_{M}\otimes((\Gamma_{N}\otimes H)\circ (H\otimes c_{H,N})\circ (\delta_{H}\otimes N)))\circ(H\otimes c_{H,M}\otimes N)\circ(\delta_{H}\otimes M\otimes N)$ {\footnotesize (by \eqref{ccclassGammaM}}
		\item[]$\hspace{0.38cm}${\footnotesize for $(N,\Gamma_{N})$)}
		\item[]$=\Gamma_{M\otimes N}\circ(H\otimes c_{H,M\otimes N})\circ(\delta_{H}\otimes M\otimes N)$ {\footnotesize (by naturality of $c$ and coassociativity of $\delta_{H}$)}.\qedhere
	\end{itemize}
\end{proof}

Having introduced the category of modules over a Hopf brace, we now address the question of whether modules over a Hopf brace can be regarded as modules over an algebra, and we provide an affirmative answer. The following theorem is the key tool to obtain the link between modules over a Hopf brace and modules over an algebra.

\begin{theorem}\label{equivalenceth}
	Let $\mathbb{H}=(H_{1},H_{2})$ be a Hopf brace in ${\sf C}$. If $(M,\varphi_{M}^{i})$ is a left $H_{i}$-module for $i=1,2$, then the following statements are equivalent:
	\begin{itemize}
		\item[(i)] $(M,\varphi_{M}^{1},\varphi_{M}^{2})$ is a left $\mathbb{H}$-module,
		\item[(ii)] $(M,\varphi_{M})$ is a left $H_{1}\sharp H_{2}$-module, where 
		\[\varphi_{M}\coloneqq \varphi_{M}^{1}\circ(H\otimes\varphi_{M}^{2}).\] 
	\end{itemize}
\end{theorem}
\begin{proof}
First note that, if  $\mathbb{H}=(H_{1},H_{2})$ is a Hopf brace in ${\sf C}$, then $(H_{1},\Gamma'_{H_{1}})$ is a left $H_{2}$-module algebra. As a consequence,  
$$H_{1}\sharp H_{2}=(H\otimes H,\eta_{H_{1}\sharp H_{2}}\coloneqq\eta_{H}\otimes\eta_{H},\mu_{H_{1}\sharp H_{2}}\coloneqq(\mu_{H}^{1}\otimes\mu_{H}^{2})\circ(H\otimes \Psi^{H_{2}}_{H_{1}}\otimes H))$$ is an algebra in ${\sf C}$, with $\Psi^{H_{2}}_{H_{1}}=(\Gamma'_{H_{1}}\otimes H)\circ(H\otimes c_{H,H})\circ(\delta_{H}\otimes H)$. 	

(i)$\implies$(ii): Condition \eqref{actioneta} for $(M,\varphi_{M})$ is straightforward and \eqref{actionprod} is obtained as follows:
	\begin{itemize}
		\itemindent=-32pt 
		\item[ ]$\hspace{0.38cm}\varphi_{M}\circ(H\otimes\varphi_{M})$
		\item[] $=\varphi_{M}^{1}\circ(H\otimes(\varphi_{M}^{2}\circ(H\otimes\varphi_{M}^{1})\circ(H\otimes H\otimes\varphi_{M}^{2})))$
		\item[] $=\varphi_{M}^{1}\circ(H\otimes(\varphi_{M}^{1}\circ (\Gamma'_{H_{1}}\otimes(\varphi_{M}^{2}\circ(H\otimes\varphi_{M}^{2})))\circ(H\otimes c_{H,H}\otimes H\otimes M)\circ(\delta_{H}\otimes H\otimes H\otimes M)))$ 
		\item[]$\hspace{0.38cm}${\footnotesize (by \eqref{compatmodHprime})}
		\item[] $=\varphi_{M}^{1}\circ(H\otimes\varphi_{M}^{2})\circ(((\mu_{H}^{1}\otimes\mu_{H}^{2})\circ(H\otimes\Psi_{H_{1}}^{H_{2}}\otimes H))\otimes M)$ {\footnotesize (by \eqref{actionprod} for $(M,\varphi_{M}^{i})$ for all $i=1,2$)}
		\item[] $=\varphi_{M}\circ(\mu_{H_{1}\sharp H_{2}}\otimes M)$.
	\end{itemize}
	
	(ii)$\implies$(i): It only remains to show that \eqref{compatmodHprime} holds. Since $(M,\varphi_{M})$ is a left $H_{1}\sharp H_{2}$-module, then 
	\begin{gather}\label{e1}
		\varphi_{M}\circ(H\otimes H\otimes\varphi_{M})=\varphi_{M}\circ(\mu_{H_{1}\sharp H_{2}}\otimes M).
	\end{gather}
	Therefore, \eqref{compatmodHprime} is obtained by composing \eqref{e1} on the right with $\eta_{H}\otimes H\otimes H\otimes\eta_{H}\otimes M$ and then using the unit property and \eqref{actioneta} for $(M,\varphi_{M}^{i})$ for all $i=1,2$.
\end{proof}

As a consequence of the previous theorem, the following functors are obtained.
\begin{lemma}
Let $\mathbb{H}=(H_{1},H_{2})$ be a Hopf brace in ${\sf C}$. There exists a functor 
\[{\sf F}\colon {}_{\mathbb{H}}{\sf Mod}\longrightarrow {}_{H_{1}\sharp H_{2}}{\sf Mod}\]
defined on objects by $${\sf F}((M,\varphi_{M}^{1},\varphi_{M}^{2}))\coloneqq(M,\varphi_{M}),$$ where $\varphi_{M}\coloneqq\varphi_{M}^{1}\circ(H\otimes\varphi_{M}^{2})$, and on morphisms by the identity.
\end{lemma}
\begin{proof}
	The previous theorem shows that ${\sf F}$ is well-defined on objects. Let us show that it is well-defined on morphisms. If $f\colon (M,\varphi_{M}^{1},\varphi_{M}^{2})\rightarrow (N,\varphi_{N}^{1},\varphi_{N}^{2})$ is a morphism of left $\mathbb{H}$-modules, then
	\begin{itemize}
		\itemindent=-32pt 
		\item[ ]$\hspace{0.38cm}f\circ \varphi_{M}$
		\item[] $=f\circ\varphi_{M}^{1}\circ(H\otimes\varphi_{M}^{2})$
		\item[] $=\varphi_{N}^{1}\circ(H\otimes (f\circ\varphi_{M}^{2}))$ {\footnotesize (by \eqref{mod_mor} for $f\colon (M,\varphi_{M}^{1})\rightarrow (N,\varphi_{N}^{1})$)} 
		\item[] $=\varphi_{N}\circ(H\otimes H\otimes f)$ {\footnotesize (by \eqref{mod_mor} for $f\colon (M,\varphi_{M}^{2})\rightarrow (N,\varphi_{N}^{2})$).}\qedhere
	\end{itemize}
\end{proof}
\begin{lemma}
	Let $\mathbb{H}=(H_{1},H_{2})$ be a Hopf brace in ${\sf C}$. There exists a functor 
	\[{\sf G}\colon {}_{H_{1}\sharp H_{2}}{\sf Mod}\longrightarrow {}_{\mathbb{H}}{\sf Mod}\]
	acting on objects by 
	\[{\sf G}((M,\varphi_{M}))\coloneqq(M,\varphi_{M}^{1},\varphi_{M}^{2}),\]
	where $\varphi_{M}^{1}\coloneqq \varphi_{M}\circ(H\otimes\eta_{H}\otimes M)$ and $\varphi_{M}^{2}\coloneqq \varphi_{M}\circ(\eta_{H}\otimes H\otimes M),$ and on morphisms by the identity.
\end{lemma}
\begin{proof}
	In the development of this proof, we will use repeatedly the following identity:
	\begin{gather}\label{psieta}
		\Psi_{H_{1}}^{H_{2}}\circ(\eta_{H}\otimes H)=H\otimes\eta_{H},
	\end{gather}
	whose justification is given below:
	\begin{itemize}
		\itemindent=-32pt 
		\item[ ]$\hspace{0.38cm}\Psi_{H_{1}}^{H_{2}}\circ(\eta_{H}\otimes H)$
		\item[]$=(\Gamma'_{H_{1}}\otimes H)\circ(H\otimes c_{H,H})\circ((\delta_{H}\circ\eta_{H})\otimes H)$
		\item[]$=((\Gamma'_{H_{1}}\circ(\eta_{H}\otimes H))\otimes H)\circ c_{H,H}\circ(\eta_{H}\otimes H)$ {\footnotesize (by the condition of coalgebra morphism for $\eta_{H}$)} 
		\item[]$=H\otimes\eta_{H}$ {\footnotesize (by \eqref{actioneta} for $(H_{1},\Gamma'_{H_{1}})$ and the naturality of $c$)}.
	\end{itemize}
	
	Let us start the proof by showing that $(M,\varphi_{M}^{1})$ is a left $H_{1}$-module. On the one hand, condition \eqref{actioneta} is straightforward by \eqref{actioneta} for $(M,\varphi_{M})$, whereas, on the other hand, condition \eqref{actionprod} is a consequence of
	\begin{itemize}
		\itemindent=-32pt 
		\item[ ]$\hspace{0.38cm}\varphi_{M}^{1}\circ(H\otimes\varphi_{M}^{1})$
		\item[] $=\varphi_{M}\circ(H\otimes\eta_{H}\otimes(\varphi_{M}\circ(H\otimes\eta_{H}\otimes M)))$
		\item[] $=\varphi_{M}\circ(((\mu_{H}^{1}\otimes\mu_{H}^{2})\circ(H\otimes (\Psi_{H_{1}}^{H_{2}}\circ(\eta_{H}\otimes H))\otimes \eta_{H}))\otimes M)$ {\footnotesize (by \eqref{actionprod} for $(M,\varphi_{M})$)}
		\item[] $=\varphi_{M}\circ(\mu_{H}^{1}\otimes\eta_{H}\otimes M)$ {\footnotesize (by \eqref{psieta} and the unit property for $H_{2}$)}
		\item[] $=\varphi_{M}^{1}\circ(\mu_{H}^{1}\otimes M).$ 
	\end{itemize}
	
	By similar arguments, $(M,\varphi_{M}^{2})$ is also a left $H_{2}$-module. Then, to conclude that ${\sf G}$ is well-defined on objects, it only remains to show that \eqref{compatmodHprime} holds for $(M,\varphi_{M}^{1},\varphi_{M}^{2})$. Developing right hand side of \eqref{compatmodHprime}, we obtain that
	\begin{itemize}
		\itemindent=-32pt 
		\item[ ]$\hspace{0.38cm}\varphi_{M}^{1}\circ(\Gamma'_{H_{1}}\otimes\varphi_{M}^{2})\circ(H\otimes c_{H,H}\otimes M)\circ(\delta_{H}\otimes H\otimes M)$
		\item[]$=\varphi_{M}\circ(H\otimes\eta_{H}\otimes M)\circ(\Gamma'_{H_{1}}\otimes(\varphi_{M}\circ(\eta_{H}\otimes H\otimes M)))\circ(H\otimes c_{H,H}\otimes M)\circ(\delta_{H}\otimes H\otimes M)$
		\item[] $=\varphi_{M}\circ(((\mu_{H}^{1}\otimes\mu_{H}^{2})\circ(H\otimes(\Psi_{H_{1}}^{H_{2}}\circ(\eta_{H}\otimes\eta_{H}))\otimes H)\circ(\Gamma'_{H_{1}}\otimes H)\circ(H\otimes c_{H,H})\circ(\delta_{H}\otimes H))\otimes M)$
		\item[]$\hspace{0.38cm}${\footnotesize (by \eqref{actionprod} for $(M,\varphi_{M})$)}
		\item[] $=\varphi_{M}\circ(((\Gamma^{\prime}_{H_{1}}\otimes H)\circ(H\otimes c_{H,H})\circ(\delta_{H}\otimes H))\otimes M)$
		\item[]$\hspace{0.38cm}${\footnotesize (by \eqref{psieta} and the unit property for $H_{1}$ and $H_{2}$)}
	\end{itemize}
	whereas, from the left hand side of \eqref{compatmodHprime} it is obtained that
	\begin{itemize}
		\itemindent=-32pt 
		\item[ ]$\hspace{0.38cm}\varphi_{M}^{2}\circ(H\otimes\varphi_{M}^{1})$
		\item[] $=\varphi_{M}\circ(\eta_{H}\otimes H\otimes(\varphi_{M}\circ(H\otimes\eta_{H}\otimes M)))$
		\item[] $=\varphi_{M}\circ(((\mu_{H}^{1}\otimes\mu_{H}^{2})\circ(\eta_{H}\otimes\Psi_{H_{1}}^{H_{2}}\otimes\eta_{H}))\otimes M)$ {\footnotesize (by \eqref{actionprod} for $(M,\varphi_{M})$)}
		\item[]$=\varphi_{M}\circ(\Psi_{H_{1}}^{H_{2}}\otimes M)$ {\footnotesize (by the unit property for $H_{1}$ and $H_{2}$)}
		\item[]$=\varphi_{M}\circ(((\Gamma'_{H_{1}}\otimes H)\circ(H\otimes c_{H,H})\circ(\delta_{H}\otimes H))\otimes M)$.
	\end{itemize}
	That is, computing both sides of \eqref{compatmodHprime} yields the same result, and hence \eqref{compatmodHprime} holds for $(M,\varphi_{M}^{1},\varphi_{M}^{2})$. 
	
	Regarding morphisms, ${\sf G}$ is well-defined because, if $f\colon (M,\varphi_{M})\rightarrow (N,\varphi_{N})$ is a morphism in ${}_{H_{1}\sharp H_{2}}{\sf Mod}$, then 
	$f\colon (M,\varphi_{M}^{1})\rightarrow (N,\varphi_{N}^{1})$ is $H_{1}$-linear. Indeed,
	\begin{itemize}
		\itemindent=-32pt 
		\item[ ]$\hspace{0.38cm}f\circ\varphi_{M}^{1}$
		\item[]$=f\circ\varphi_{M}\circ(H\otimes\eta_{H}\otimes M)$
		\item[]$=\varphi_{N}\circ(H\otimes\eta_{H}\otimes f)$ {\footnotesize (by \eqref{mod_mor} for $f\colon (M,\varphi_{M})\rightarrow (N,\varphi_{N})$)}
		\item[]$=\varphi_{N}^{1}\circ (H\otimes f)$.
	\end{itemize}
	Similar arguments can be followed to prove that $f\colon (M,\varphi_{M}^{2})\rightarrow (N,\varphi_{N}^{2})$ is also $H_{2}$-linear.
\end{proof}
The functors constructed in the previous lemmas yield the required isomorphism of categories.
\begin{theorem}\label{mainth}
	Let $\mathbb{H}$ be a Hopf brace in ${\sf C}$. The categories ${}_{\mathbb{H}}{\sf Mod}$ and ${}_{H_{1}\sharp H_{2}}{\sf Mod}$ are isomorphic.
\end{theorem}
\begin{proof}
	First, it is easy to obtain that ${\sf G}\circ{\sf F}={\sf id}_{{}_{\mathbb{H}}{\sf Mod}}$. Indeed, if $(M,\varphi_{M}^{1},\varphi_{M}^{2})\in{}_{\mathbb{H}}{\sf Mod}$, then
	\[({\sf G}\circ{\sf F})((M,\varphi_{M}^{1},\varphi_{M}^{2}))={\sf G}((M,\varphi_{M}^{1}\circ(H\otimes\varphi_{M}^{2})))=(M,\varphi_{M}^{1},\varphi_{M}^{2}),\]
	where the last equality is direct by \eqref{actioneta} for $(M,\varphi_{M}^{1})$ and $(M,\varphi_{M}^{2})$.
	
	Converserly, if $(M,\varphi_{M})\in{}_{H_{1}\sharp H_{2}}{\sf Mod}$, then
	\begin{align*}
		({\sf F}\circ{\sf G})((M,\varphi_{M}))=\,&{\sf F}((M,\varphi_{M}\circ(H\otimes\eta_{H}\otimes M),\varphi_{M}\circ(\eta_{H}\otimes H\otimes M)))\\=\,&(M,\varphi_{M}\circ(H\otimes\eta_{H}\otimes(\varphi_{M}\circ(\eta_{H}\otimes H\otimes M)))),
	\end{align*}
	and
	\begin{itemize}
		\itemindent=-32pt 
		\item[ ]$\hspace{0.38cm}\varphi_{M}\circ(H\otimes\eta_{H}\otimes(\varphi_{M}\circ(\eta_{H}\otimes H\otimes M)))$
		\item[]$=\varphi_{M}\circ(((\mu_{H}^{1}\otimes\mu_{H}^{2})\circ(H\otimes(\Psi^{H_{2}}_{H_{1}}\circ(\eta_{H}\otimes\eta_{H}))\otimes H))\otimes M)$ {\footnotesize (by \eqref{actionprod} for $(M,\varphi_{M})$)}
		\item[]$=\varphi_{M}$ {\footnotesize (by \eqref{psieta} and the unit property for $H_{1}$ and $H_{2}$)}.
	\end{itemize}
	Therefore, ${\sf F}\circ{\sf G}={\sf id}_{H_{1}\sharp H_{2}},$ what concludes the proof.
\end{proof}

\begin{corollary}\label{mainthcor}
	Let $\mathbb{H}$ be a profinite Hopf brace in ${\sf C}$. The categories ${}_{\mathbb{H}}{\sf Mod}$, ${}_{H_{1}\sharp H_{2}}{\sf Mod}$ and $\;_{H_1}{\sf M}^{H_{2}^{\ast}}$  are isomorphic.
\end{corollary}
\begin{proof}
The proof follows from the previous theorem and Theorem \ref{smash}.
\end{proof}

 Hence, as a consequence of the previous corollary, modules over a profinite Hopf brace $\mathbb{H}$ can be seen as Doi-Hopf modules, more concretely,  as left-right $(H_{1}, H_{2}^{\ast})$-Doi-Hopf modules.  

On the other hand, note that, as was pointed out in Example \ref{smashHA}, if $(H_{1},\Gamma'_{H_{1}})$ lies in the cocommutativity class of $H_{2}$, then $H_{1}\sharp H_{2}$ is a Hopf algebra in ${\sf C}$. Therefore, under such condition, modules over $\mathbb{H}$ are identified with modules over the Hopf algebra $H_{1}\sharp H_{2}$. 
\begin{example}
	In Example \ref{skewbraces} we have seen that, if ${\sf C}={\sf Set}$, any Hopf brace is a skew brace. Hence, as a particular case of Theorem \ref{mainth} with ${\sf C}={\sf Set}$, the following is obtained: If $(G,\cdot,\circ)$ is a skew brace, then the category of modules over $(G,\cdot,\circ)$ is isomorphic to the category of modules over the semidirect product group $G_{\cdot}\ltimes G_{\circ}$, $${}_{(G,\cdot,\circ)}{\sf Mod}\cong{}_{G_{\cdot}\ltimes G_{\circ}}{\sf Mod},$$
	where an object in ${}_{(G,\cdot,\circ)}{\sf Mod}$ is a triple $(M,\rhd,\blacktriangleright)$ such that $(M,\rhd)$ is a module over the group $G_{\cdot}$ and $(M,\blacktriangleright)$ is a module over the group $G_{\circ}$ satisfying the following identity: 
	\[g\blacktriangleright(h\rhd m)=((g\circ h)\cdot g^{-1})\rhd(g\blacktriangleright m),\textnormal{ for all }g,h\in G\textnormal{ and }m\in M.\]
	
As we observed in Example \ref{galg},  if $(G,\cdot,\circ)$ is a skew brace, then  $\mathbb{K}[\mathbb{G}]=(\mathbb{K}[G_{\cdot}],\mathbb{K}[G_{\circ}])$ forms a Hopf brace in $\;_{\mathbb{K}}{\sf Vect}$. Therefore, again by Theorem \ref{mainth}, we obtain that the category of modules over $\mathbb{K}[\mathbb{G}]$ is isomorphic to the category of modules over the Hopf algebra $\mathbb{K}[G_{\cdot}]\sharp \mathbb{K}[G_{\circ}]$. As was pointed out in \cite[Example 3.14]{Maj} (see also \cite{Tak}), $\mathbb{K}[G_{\cdot}]\sharp \mathbb{K}[G_{\circ}]=\mathbb{K}[G_{\cdot}\ltimes G_{\circ}]$. Thus, we obtain the following chain of categorical isomorphisms:
\[{}_{\mathbb{K}[\mathbb{G}]}{\sf Mod}\cong {}_{\mathbb{K}[G_{\cdot}]\sharp \mathbb{K}[G_{\circ}]}{\sf Mod}={}_{\mathbb{K}[G_{\cdot}\ltimes G_{\circ}]}{\sf Mod}.\]

These chain of isomorphisms extend \cite[Remarks 2.2 and 3.2]{KOZTSANG}, where the authors only construct a functor from ${}_{\mathbb{K}[\mathbb{G}]}{\sf Mod}$ to ${}_{\mathbb{K}[G_{\cdot}\ltimes G_{\circ}]}{\sf Mod}$, but they do not find the equivalence.

Now, by Corollary \ref{mainthcor}, if we assume that $G$ is a finite group, then it is also obtained that
\[{}_{\mathbb{K}[\mathbb{G}]}{\sf Mod}\cong {}_{\mathbb{K}[G_{\cdot}\ltimes G_{\circ}]}{\sf Mod}\cong {}_{\mathbb{K}[G_{\cdot}]}{\sf M}^{{\mathbb{K}[G_{\circ}]}^{\ast}},\]
extending the chain above one step further.
\end{example}
\section{On the notion of modules over a Hopf brace in Zhu's sense}
\label{comparison}

An alternative definition of left module over a Hopf brace, different from the one provided in the previous section, was given by Zhu in \cite[Definition 3.1]{ZHU}. To avoid confusion with the notion considered in Section \ref{sec_equivalence}, we will henceforth refer to these structures as ${\sf Z}$-modules, where this name comes from the fact that these are modules in the sense of Zhu. The definition is as follows:
\begin{definition}	\label{ZHUdef}
	Let $\mathbb{H}=(H_{1},H_{2})$ be a Hopf brace in ${\sf C}$. A left ${\sf Z}$-module over $\mathbb{H}$ is a triple $(M,\varphi_{M}^{1},\varphi_{M}^{2})$ where $(M,\varphi_{M}^{1})$ is a left $H_{1}$-module, $(M,\varphi_{M}^{2})$ is a left $H_{2}$-module and the equality 	
\begin{align}\label{realZhu}
		\begin{split}
		&(\varphi_{M}^{2}\otimes H)\circ (H\otimes c_{H,M})\circ(\delta_{H}\otimes\varphi_{M}^{1})\\=\,&(\varphi_{M}^{1}\otimes H)\circ(\mu_{H}^{2}\otimes((\Gamma_{M}\otimes H)\circ (H\otimes c_{H,M})\circ ((c_{H,H}\circ\delta_{H})\otimes M)))\\&\circ(H\otimes c_{H,H}\otimes M)\circ (\delta_{H}\otimes H\otimes M)
		\end{split}
\end{align}  
	holds.
	
	A morphism of left ${\sf Z}$-modules over $\mathbb{H}$ is defined as in the case of Definition \ref{RGONdef} and we will denote the corresponding category  by ${}_{\mathbb{H}}{\sf Mod}^{{\sf Z}}$. 
\end{definition}

As was proved in \cite[Lemma 3.2]{ZHU}, given a Hopf brace $\mathbb{H}=(H_{1},H_{2})$ in ${\sf C}$, if $(M,\varphi_{M}^{1})$ is a left $H_{1}$-module and $(M,\varphi_{M}^{2})$ is a left $H_{2}$-module, then equality \eqref{realZhu} is equivalent to equalities \eqref{compatmodH} and  
\begin{align}\label{condModZhu}
\begin{split}
&(\varphi_{M}^{2}\otimes H)\circ(H\otimes c_{H,M})\circ(\delta_{H}\otimes M)\\=\,&(\varphi_{M}^{1}\otimes H)\circ(H\otimes((\Gamma_{M}\otimes H)\circ (H\otimes c_{H,M})\circ ((c_{H,H}\circ\delta_{H})\otimes M)))\circ(\delta_{H}\otimes M).
\end{split}
\end{align}  
	
As a consequence, ${}_{\mathbb{H}}{\sf Mod}^{{\sf Z}}$ is a full subcategory of ${}_{\mathbb{H}}{\sf Mod}$. In general, these two categories are different since $(H,\mu_{H}^{1},\mu_{H}^{2})$ is an example of module over $\mathbb{H}$ (see Example \ref{ex_HmodHbrace}), but it is not in the sense of Definition \ref{ZHUdef} because the triple $(H,\mu_{H}^{1},\mu_{H}^{2})$ does not satisfy condition \eqref{condModZhu}.  However in \cite[Remark 3.3]{ZHU}, Zhu observes that under cocommutativity conditions, equality \eqref{condModZhu} is trivial, so this would mean that both notions would be equivalent. However, it is not necessary to be that restrictive. The following result shows the conditions under which an object in ${}_{\mathbb{H}}{\sf Mod}$ satisfies condition \eqref{condModZhu}, and therefore it is an object in the category ${}_{\mathbb{H}}{\sf Mod}^{{\sf Z}}$. 

\begin{theorem}\label{firstM}
	Let ${\mathbb H}$ be a Hopf brace in ${\sf C}$. Given $(M,\varphi_{M}^{1},\varphi_{M}^{2})$ an object in the category ${}_{\mathbb{H}}{\sf Mod}$, then  $(M,\varphi_{M}^{1},\varphi_{M}^{2})$ is an object in ${}_{\mathbb{H}}{\sf Mod}^{{\sf Z}}$ if and only if $(M,\Gamma_{M})$ belongs to the cocommutativity class of $H_{2}$.
\end{theorem}
\begin{proof}
	Let us assume that \eqref{ccclassGammaM} holds for $(M,\Gamma_{M})$. Then, the following is obtained:
	\begin{itemize}
		\itemindent=-32pt 
		\item[]$\hspace{0.38cm}(\varphi_{M}^{1}\otimes H)\circ(H\otimes((\Gamma_{M}\otimes H)\circ (H\otimes c_{H,M})\circ ((c_{H,H}\circ\delta_{H})\otimes M)))\circ(\delta_{H}\otimes M)$
		\item[]$=(\varphi_{M}^{1}\otimes H)\circ(H\otimes((\Gamma_{M}\otimes H)\circ(H\otimes c_{H,M})\circ(\delta_{H}\otimes M)))\circ(\delta_{H}\otimes M)$ {\footnotesize (by \eqref{ccclassGammaM} for $(M,\Gamma_{M})$)}
		\item[]$=((\varphi_{M}^{1}\circ(H\otimes \Gamma_{M})\circ(\delta_{H}\otimes M))\otimes H)\circ(H\otimes c_{H,M})\circ(\delta_{H}\otimes M)$ {\footnotesize (by coassociativity of $\delta_{H}$)}
		\item[]$=(\varphi_{M}^{2}\otimes H)\circ (H\otimes c_{H,M})\circ(\delta_{H}\otimes M)$ {\footnotesize (by \eqref{varphi2expression})},
	\end{itemize}
	that is, \eqref{condModZhu} holds.
	
Conversely, if \eqref{condModZhu} holds,  we obtain 
\begin{itemize}
	\itemindent=-32pt 
	\item[]$\hspace{0.38cm} (\Gamma_{M}\otimes H)\circ (H\otimes c_{H,M})\circ (\delta_{H}\otimes M)$
	\item[]$= (\varphi_{M}^{1}\otimes H)\circ (\lambda_{H}^{1}\otimes ((\varphi_{M}^{2}\otimes H)\circ (H\otimes c_{H,M})\circ (\delta_{H}\otimes M)))\circ (\delta_{H}\otimes M)$  {\footnotesize (by coassociativitiy of $\delta_{H}$)}
	\item[]$= (\varphi_{M}^{1}\otimes H)\circ (\lambda_{H}^{1}\otimes ((\varphi_{M}^{1}\otimes H)\circ(H\otimes((\Gamma_{M}\otimes H)\circ (H\otimes c_{H,M})\circ ((c_{H,H}\circ\delta_{H})\otimes M)))\circ(\delta_{H}\otimes M)))$ 
	\item[]$\hspace{0.38cm} \circ (\delta_{H}\otimes M) ${\footnotesize (by \eqref{condModZhu})}.
	\item[]$= (\varphi_{M}^{1}\otimes H)\circ ((\lambda_{H}^1\ast {\rm id}_{H})\otimes ((\Gamma_{M}\otimes H)\circ (H\otimes c_{H,M})\circ ((c_{H,H}\circ \delta_{H})\otimes M)))\circ (\delta_{H}\otimes M)$  {\footnotesize (by }
	\item[]$\hspace{0.38cm}$  {\footnotesize coassociativitiy of $\delta_{H}$ and \eqref{actionprod} for $(M,\varphi_{M}^{1})$)}
	\item[]$= (\Gamma_{M}\otimes H)\circ (H\otimes c_{H,M})\circ ((c_{H,H}\circ \delta_{H})\otimes M)$  {\footnotesize (by \eqref{antipode} for $H_{1}$, the counit property and \eqref{actioneta} for $(M,\varphi_{M}^{1})$),}
\end{itemize}
i.e., $(M,\Gamma_{M})$ is in the cocommutativity class of $H_{2}$.
\end{proof}
Therefore, modules in the sense of Zhu are nothing but modules for which $\Gamma_{M}$ lies in the cocommutativity class, providing a significantly simpler characterization of the category ${}_{\mathbb{H}}{\sf Mod}^{{\sf Z}}$ than that given by conditions \eqref{realZhu} and \eqref{condModZhu}. Thus, as an immediate  consequence of the previous theorem, we have the following corollaries:
\begin{corollary} 
	\label{firstMH}
	Let ${\mathbb H}$ be a  Hopf brace in ${\sf C}$. Then, ${}_{\mathbb{H}}{\sf Mod}^{{\sf Z}}={}_{\mathbb{H}}{\sf Mod}^{{\sf cc}}$.
\end{corollary}

\begin{corollary} 
	\label{firstH}
	Let ${\mathbb H}$ be a Hopf brace in ${\sf C}$. The triple $(H,\mu_{H}^{1},\mu_{H}^{2})$  is an object in ${}_{\mathbb{H}}{\sf Mod}^{{\sf Z}}$ if and only if $(H_{1},\Gamma_{H_{1}})$ belongs to the cocommutativity class of $H_{2}$.
\end{corollary}
\begin{remark}
	The previous result was previously obtained in \cite[Remark 2.9]{VGRRProj}.
\end{remark}

\begin{corollary}
	\label{firstMH1}
	Let ${\mathbb H}$ be a cocommutative Hopf brace. Then  ${}_{\mathbb{H}}{\sf Mod}^{{\sf Z}}={}_{\mathbb{H}}{\sf Mod}$.
\end{corollary}
\begin{remark}
	The content of the previous corollary corresponds to the observation made in \cite[Remark 2.1]{RGON}.
\end{remark}

\section*{Funding}

The  authors are supported by  Ministerio de Ciencia e Innovaci\'on. Agencia Estatal de Investigaci\'on (Spain) grants no. PID2020-115155GB-I00 and PID2024-15502NB-I00 (European FEDER support included, UE).

Moreover, B. Ramos Pérez is funded by Xunta de Galicia through the Competitive Reference Groups (GRC) grant no. ED431C 2023/31,  and the fellowship grant no. ED481A-2023-023.

%
%
%
%

\bibliographystyle{amsalpha}

\end{document}